\documentclass{amsart}[11pt]

\usepackage{times,amssymb}
\usepackage{amsmath}
\usepackage{amssymb}
\usepackage{amsthm}
\usepackage{amscd}
\usepackage{stmaryrd}
\usepackage{color}
\usepackage[dvipdfmx]{graphicx}
\usepackage{float}
\usepackage{comment}

\newtheorem{thm}{Theorem}[section]
\newtheorem{cor}[thm]{Corollary}
\newtheorem{lem}[thm]{Lemma}
\newtheorem{prop}[thm]{Proposition}

\theoremstyle{definition}

\newtheorem{rem}[thm]{Remark}
\newtheorem{example}[thm]{Example}

\newtheorem{quest}[thm]{Question}

\numberwithin{equation}{section}

\newcommand{\spn}{\mathrm{span\,}}

\def\LL{\mathcal L}

\begin{document}
\title[Kauffman-Jones polynomial of a curve]
{Kauffman-Jones polynomial of a curve on a surface}

\author{Shinji Fukuhara}
\address{Department of Mathematics, Tsuda College, Tsuda-machi 2-1-1,
  Kodaira-shi, Tokyo 187-8577, Japan}
\email{fukuhara@tsuda.ac.jp}

\author{Yusuke Kuno}
\address{Department of Mathematics, Tsuda College, Tsuda-machi 2-1-1,
  Kodaira-shi, Tokyo 187-8577, Japan}
\email{kunotti@tsuda.ac.jp}

\date{}

\begin{abstract}
We introduce a Kauffman-Jones type polynomial $\LL_{\gamma}(A)$ 
for a curve $\gamma$ on an oriented surface, whose endpoints are on the boundary of the surface.
The polynomial $\LL_{\gamma}(A)$ is a Laurent polynomial 
in one variable $A$ and is an invariant of the homotopy class of $\gamma$.
As an application, we obtain an estimate in terms of the span of $\LL_{\gamma}(A)$ for the minimum
self-intersection number of a curve within its homotopy class.
We then give a chord diagrammatic description of $\LL_{\gamma}(A)$ and show some computational results on the span of $\LL_{\gamma}(A)$.
\end{abstract}

\maketitle

\section{Introduction}

Let $S$ be an oriented $C^{\infty}$-surface with non-empty boundary $\partial S$.
By a \emph{curve} on $S$, we mean a $C^{\infty}$-immersion $\gamma$ from the unit interval $I=[0,1]$ to $S$, which has only transverse double points as its singularities and satisfies
$\gamma^{-1}(\partial S)=\{ 0,1\}$ with $\gamma(0)\neq \gamma(1)$.

In this article, we consider curves on $S$ from the view point of virtual knots \cite{Ka97} or equivalently, abstract link diagrams \cite{KaKa}, with emphasis on their invariants coming from the Kauffman bracket \cite{KA1}.
More concretely, we introduce Laurent polynomials
$\langle D_{\gamma} \rangle$ and $\LL_{\gamma}(A)$ 
in one variable $A$.
We show that the span of these polynomials can be used for
estimating the number of double points of $\gamma$.
In fact, the polynomials $\langle D_{\gamma} \rangle$ and $\LL_{\gamma}(A)$ depend only on 
combinatorics of the image of the curve $\gamma$
in its regular neighborhood in $S$.
Based on this fact, we then give a chord diagrammatic description of
these polynomials.
An advantage of being free from the ambient surface $S$ is
that it becomes easy to provide and compute examples.
In \S 4, we show some computational results on the span
of $\langle D_{\gamma} \rangle$ from this point of view.

In the rest of this section,
we describe main constructions and results.
Some proofs will be postponed to \S 2.

We begin with terminology.
Let $X$ be a compact 1-manifold.
Namely, $X$ is a disjoint union of finitely many $I$'s and circles:
$$X=I\sqcup\dots\sqcup I\sqcup S^1\sqcup\dots\sqcup S^1.$$
A $C^{\infty}$-immersion $f:X\to S$ is called \emph{generic} if it has only transeverse double points as its singularities, $f^{-1}(\partial S)=\partial X$, and $f|_{\partial X}$ is injective.
A \emph{generalized link diagram} on $S$ is a subset of $S$ of the form $D=f(X)$ for some generic immersion $f: X\to S$, endowed with a choice of crossing to each double point of $D$. See Figures \ref{crossing} and \ref{gld}.

\begin{figure}
\begin{minipage}{0.45\hsize}
  \centering
  \includegraphics[width=12em]{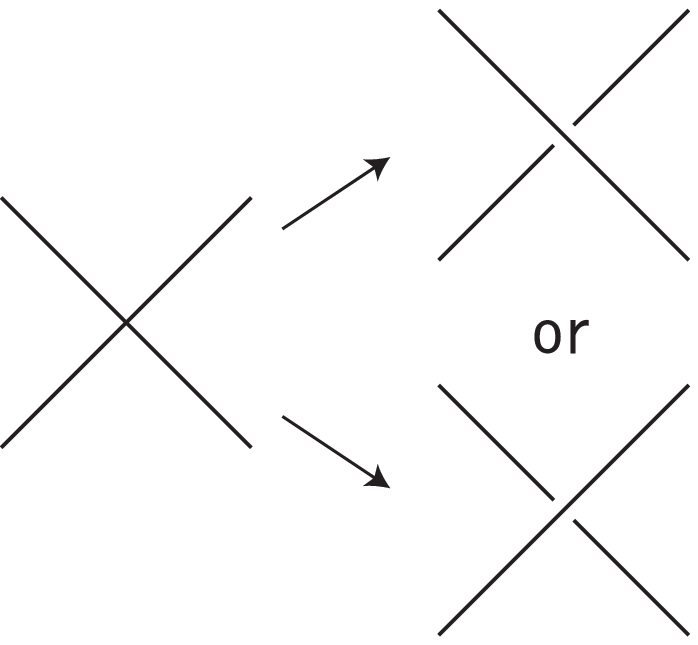}
  \caption{a choice of under- and overcrossing}
  \label{crossing}
\end{minipage}
\hspace{1cm}
\begin{minipage}{0.45\hsize}
  \centering
  \includegraphics[width=8em]{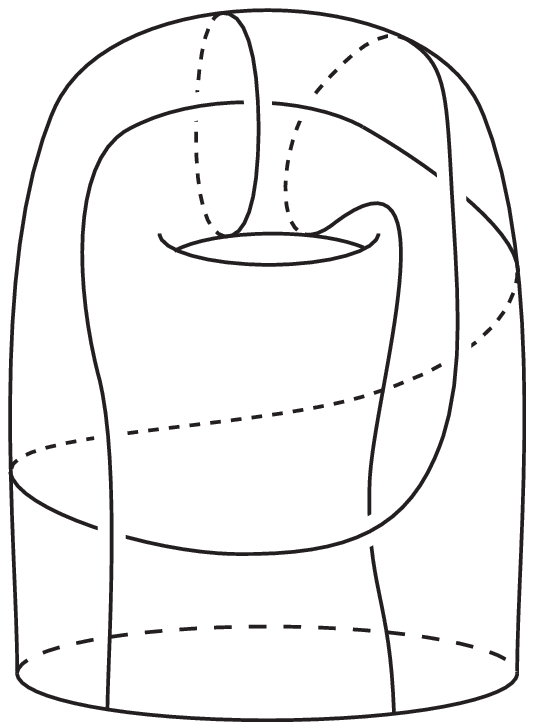}
  \caption{an example of a generalized link diagram}
  \label{gld}
\end{minipage}
\end{figure}

Two generalized link diagrams $D$ and $D'$ are called \emph{equivalent} (resp. \emph{regularly equivalent}) if $D$ is transformed into $D'$ by a finite sequence of ambient isotopies of $S$ relative to $\partial S$, and the three Reidemeister
moves $R_1$, $R_2$, and $R_3$ (resp. $R_2$ and $R_3$) shown in Figure \ref{Reidemeister}.
We write $D\sim D'$ (resp. $D\sim_r D'$) when $D$ is equivalent (resp. regularly equivalent) to $D'$.

\begin{figure}
 \centering
\includegraphics[width=30em]{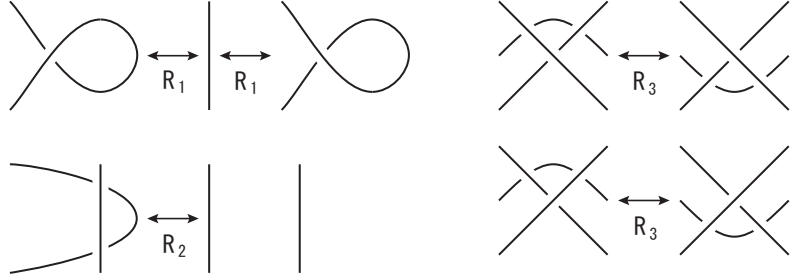}
  \caption{Reidemeister moves}
  \label{Reidemeister}
\end{figure}

Let $\gamma$ be a curve on $S$.
For each double point $p$ of $\gamma$, there is a neighborhood $U$ of $p$ such that $U\cap \gamma(I)$ consists of two arcs $J_1$ and $J_2$ intersecting at $p$, and $J_1$ is traversed first when we go along $\gamma$ from $\gamma(0)$.
Then we replace $p$ with a crossing with $J_1$ being overcrossing 
(see Figure \ref{doublept}).
Let $D_{\gamma}$ denote the generalized link diagram on $S$ obtained in this way. In other words, $D_{\gamma}$ is the projection diagram in the usual sense of the embedding $I\to S\times I, t\mapsto (\gamma(t),1-t)$ by the projection $S\times I\to S\times \{0\} \cong S, (x,t)\mapsto x$.

\begin{figure}
  \centering
  \includegraphics[width=25em]{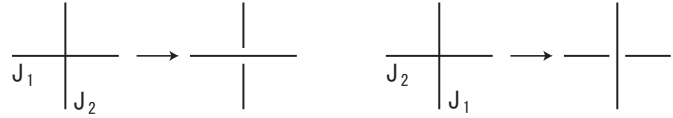}
  \caption{replacing double points with crossings}
  \label{doublept}
\end{figure}

The following fact is crucial in our argument:
\begin{thm}\label{thm1.1}
Suppose that  two curves $\gamma$ and $\gamma'$ on $S$ are homotopic (resp. regularly homotopic) relative to $\partial S$.
Then $D_{\gamma}\sim D_{\gamma'}$ (resp. $D_{\gamma}\sim_r D_{\gamma'}$).
\end{thm}

The \emph{Kauffman bracket} \cite{KA1} is extended to link diagrams on surfaces \cite{IK94}.
This extension is straightforward and applies to our generalized link diagrams also.
For the sake of definiteness, let us recall the construction.
Let $D$ be a generalized link diagram on $S$.
We can split $D$ at each crossing in two ways. 
We will distinguish these splittings as a type A splitting and a type B splitting, respectively (see Figures \ref{splitting1} and \ref{splitting2}, according to the orientation of $S$).
A \emph{state} of $D$ is a choice of splitting type for each crossing of $D$.
For a state $s$ of $D$, let $D(s)$ be the compact 1-submanifold of $S$ obtained by splitting $D$ by $s$.
If $D$ has $n$ crossings, there are $2^n$ states of $D$.

\begin{figure}
\begin{minipage}{0.45\hsize}
  \centering
  \includegraphics[width=12em]{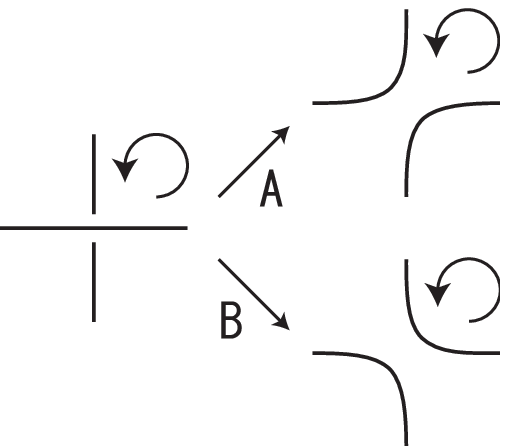}
  \caption{splitting with an orientation}
  \label{splitting1}
\end{minipage}
\hspace{1cm}
\begin{minipage}{0.45\hsize}
  \centering
  \includegraphics[width=12em]{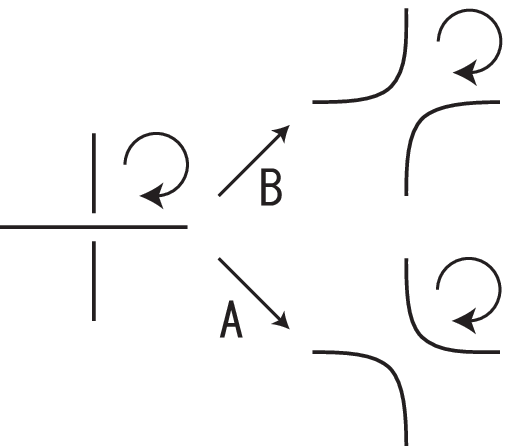}
  \caption{splitting with the other orientation}
  \label{splitting2}
\end{minipage}
\end{figure}

To each state $s$ of $D$, we assign the following three numbers:
\begin{align*}
& \text{$\alpha(s):=$ the number of type A splittings}, \\
& \text{$\beta(s):=$ the number of type B splittings}, \\
& \text{$\mu(s):=$ the number of connected components of $D(s)$.}
\end{align*}
Then we define the \emph{bracket polynomial} of $D$ by
$$
\langle D\rangle:=
  \sum_{s}A^{\alpha(s)-\beta(s)}(-A^2-A^{-2})^{\mu(s)-1},
$$
where $s$ runs over all states of $D$.

\begin{figure}
  \centering
  \includegraphics[height=2cm]{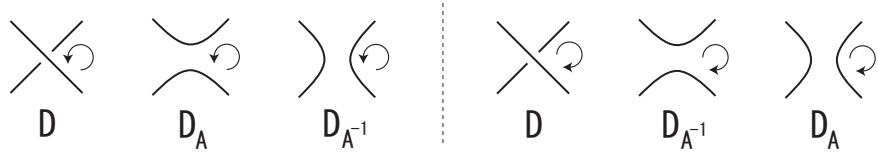}
  \caption{three diagrams}
  \label{split}
\end{figure}

A basic property of the bracket polynomial is the following skein relation, whose proof is the same as that of the classical case \cite{KA1}.

\begin{lem}
\label{lem1.2}
Let $D$ be a generalized link diagram on $S$.
\begin{enumerate}
\item
Pick a crossing of $D$ and consider the two splittings of it as shown in Figure \ref{split}.
Then
\[
\langle D\rangle=A\langle D_A\rangle+A^{-1}\langle D_{A^{-1}}\rangle.
\]
\item
Let $T$ be a generalized link diagram which is connected and has no crossing.
\begin{enumerate}
\item
We have $\langle T \rangle=1$.
\item
If $D$ and $T$ are disjoint, then
$\langle D\sqcup T \rangle=(-A^2-A^{-2})\langle D\rangle$.
\end{enumerate}
\end{enumerate}
\end{lem}

Assume that a generalized link diagram $D=f(X)$ is \emph{oriented}.
That is, $X$ is oriented and $D$ inherits this orientation.
For instance, if $\gamma$ is a curve on $S$, then $D_{\gamma}$ can be oriented from the natural orientation of $I$.
Let $w(D)$ denote the \emph{writhe number} of $D$. That is,
$$w(D):=\sum_{p}\varepsilon_p,$$
where $p$ runs over all crossings of $D$ and $\varepsilon_p\in \{ \pm 1\}$ is the sign of the crossing at $p$ (see Figure \ref{sign}).
Then we define the \emph{Kauffman-Jones polynomial} of $D$ by
$$\LL_{D}(A):=(-A)^{-3w(D)}\langle D\rangle.$$

\begin{figure}
  \centering
  \includegraphics[height=2cm]{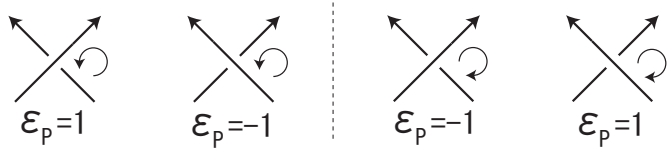}
  \caption{signs of crossings}
  \label{sign}
\end{figure}

The following result is an analogy of the result for ordinary link diagrams given by Kauffman \cite{KA1},
where Lemma \ref{lem1.2} played a central role.
His argument can also be applied to the case of generalized link diagrams, so we omit the proof.

\begin{thm}\label{thm1.3}
Let $D$ and $D'$ be generalized link diagrams on $S$.
\begin{enumerate}
\item
If $D$ and $D'$ are regularly equivalent,
$\langle D\rangle=\langle D'\rangle$.
\item
Assume further that $D$ and $D'$ are oriented.
If $D$ and $D'$ are equivalent, $\LL_D(A)=\LL_{D'}(A)$.
\end{enumerate}
\end{thm}

To simplify notation, we denote
$\LL_{\gamma}(A):=\LL_{D_{\gamma}}(A)$ for a curve $\gamma$.
Combining Theorems \ref{thm1.1} and \ref{thm1.3}, we obtain
\begin{thm}\label{thm1.4}
Let $\gamma$ and $\gamma'$ be curves on $S$.
\begin{enumerate}
\item
If $\gamma$ and $\gamma'$ are regularly homotopic relative to $\partial S$, then $\langle D_{\gamma} \rangle=\langle D_{\gamma'} \rangle$.
\item
If $\gamma$ and $\gamma'$ are homotopic relative to $\partial S$, then
$\LL_{\gamma}(A)=\LL_{\gamma'}(A)$.
\end{enumerate}
\end{thm}

For a Laurent polynomial $f(A)\in \mathbb{Z}[A,A^{-1}]$, the \emph{span} of $f$, denoted by $\spn f$, is defined to be the difference of the maximal and the minimal degrees of $f$.
Note that $\spn \langle D_{\gamma}\rangle=\spn \LL_{\gamma}(A)$ for any curve $\gamma$.
We denote by $d(\gamma)$ the number of double points of a curve $\gamma$.
Then we have the following estimate for $d(\gamma)$, which is analogous to \cite{MU1} and \cite{Th87}.

\begin{thm}\label{thm1.5} 
For a curve $\gamma$ on $S$, it holds that
\begin{equation}\label{estim}
\spn \langle D_{\gamma}\rangle\leq 4d(\gamma).
\end{equation}
\end{thm}

We define the minimum self-intersection number $c(\gamma)$ of
a curve $\gamma$ by
$$ c(\gamma):=\min\,\{\,d(\gamma')\,|\,
  \text{$\gamma'$ is a curve on $S$ homotopic to $\gamma$ relative to $\partial S$}\}.$$

\begin{cor}\label{cor1.6} 
For any curve $\gamma$ on $S$, it holds that
$$\frac{\spn \langle D_{\gamma}\rangle}{4}\leq  c(\gamma).$$
\end{cor}

We give examples of using Corollary \ref{cor1.6}
for estimating $c(\gamma)$.

\begin{figure}
\begin{minipage}{0.45\hsize}
  \centering
  \includegraphics[width=8em]{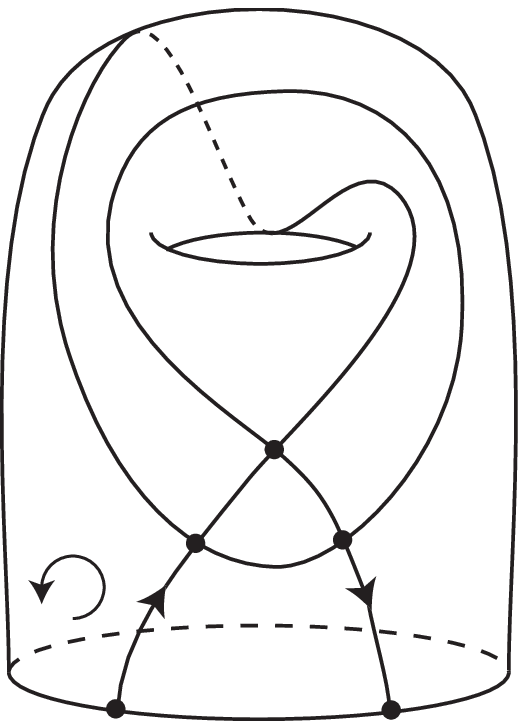}
  \caption{a curve $\gamma_1$ on a punctured torus}
  \label{fig2}
\end{minipage}
\hspace{1cm}
\begin{minipage}{0.45\hsize}
  \centering
  \includegraphics[width=8em]{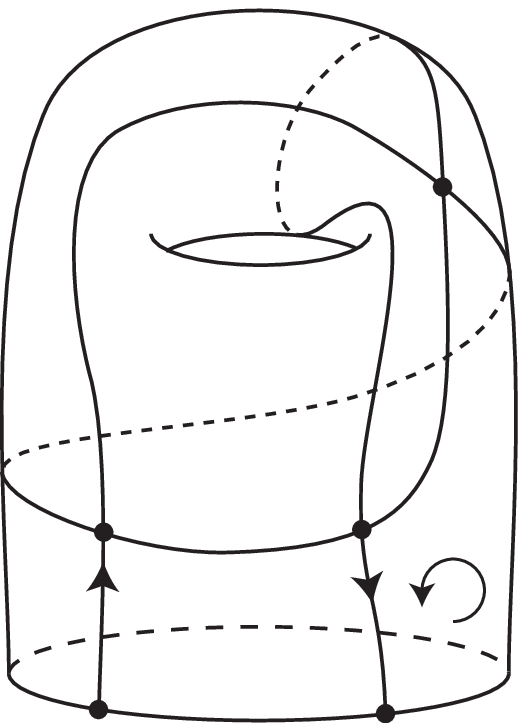}
   \caption{a curve $\gamma_2$ on a punctured torus}
  \label{fig3}
\end{minipage}
\end{figure}

\begin{example}
Let $\gamma_1$ be the curve shown in Figure \ref{fig2}.
The bracket polynomial of $\gamma_1$ is
$$\langle D_{\gamma_1}\rangle=A-A^{-3}-A^{-5}.$$
We see that
$\spn \langle D_{\gamma_1}\rangle=6$
and $6/4\leq c(\gamma_1)$.
Hence we obtain $2\leq c(\gamma_1)\leq 3$.
\end{example}

\begin{example}
Let $\gamma_2$ be the curve shown in Figure \ref{fig3}.
The bracket polynomial of $\gamma_2$ is
$$\langle D_{\gamma_2} \rangle=-A^5+A+A^{-1}-A^{-3}-A^{-5}.$$
Since
$\spn \langle D_{\gamma_2}\rangle=10$,
we have $10/4\leq c(\gamma_2)$.
Therefore, $c(\gamma_2)=3$.
\end{example} 

\section{Proofs of Theorems \ref{thm1.1} and \ref{thm1.5}}
In this section, we prove Theorems \ref{thm1.1} and \ref{thm1.5}.

\begin{proof}[Proof of Theorem \ref{thm1.1}]
If two curves $\gamma$ and $\gamma'$ are homotopic relative to $\partial S$, 
then $\gamma$ is transformed into $\gamma'$ by using a finite sequence of ambient isotopies of $S$ relative to $\partial S$ and the three local moves $\omega_1$, $\omega_2$, $\omega_3$, shown in Figure \ref{Reidemeister-omega}. See e.g., \cite{Go86} Lemma 5.6.

\begin{figure}
  \centering
  \includegraphics[width=30em]{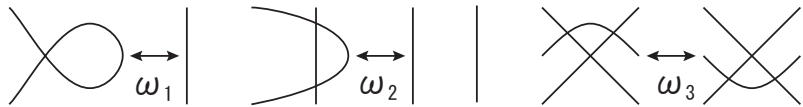}
  \caption{Reidemeister moves of a curve $\gamma$}
  \label{Reidemeister-omega}
\end{figure}

\begin{figure}
  \centering
  \includegraphics[height=3cm]{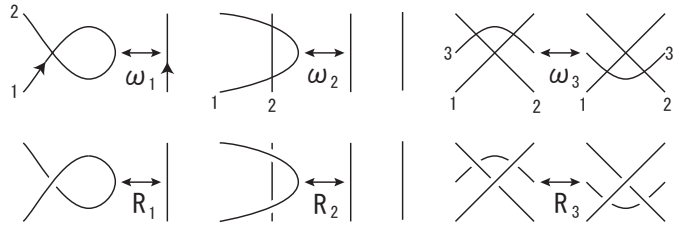}
  \caption{Reidemeister moves of $\gamma$ and $D_{\gamma}$}
  \label{Reidemeister-omegaR}
\end{figure}

It is easily seen that if $\gamma$ is transformed into $\gamma'$ by 
$\omega_i$ ($i=1,2,3$), then $D_{\gamma}$ can be transformed into $D_{\gamma'}$ 
by $R_i$ ($i=1,2,3$) respectively
(see Figure \ref{Reidemeister-omegaR}).
This completes the proof.
\end{proof}

Next, we prove Theorem \ref{thm1.5}.
Recall that a generalized link diagram $D$ has the form
$D=f(X)$ for a generic immersion $f: X\to S$, endowed with a choice of crossing to each double point.
We say that $D$ is \emph{connected}
if it is connected as a subset of $S$.
Let $d(D)$ be the number of crossings of $D$.

Let us consider the following condition for a generalized link diagram
$D=f(X)$:
\begin{equation}\label{Xform}
\text{the number of connected components of $X$
homeomorphic to $I$ is at most one.}
\end{equation}
Since $D_{\gamma}$ is connected for any curve $\gamma$, Theorem \ref{thm1.5} is a special case of the following:

\begin{prop}\label{prop4.1} 
Let $D$ be a connected generalized link diagram
satisfying condition $(\ref{Xform})$.
Then it holds that
$$\spn  \langle D \rangle \leq 4d(D).$$
\end{prop}
\begin{proof}
The bracket polynomial of $D$ is written as
$$\langle D\rangle=
  \sum_{s} \langle D|s \rangle \delta^{\mu(s)-1},$$
where $s$ runs over all states of $D$ and we set $\langle D|s\rangle:=A^{\alpha(s)-\beta(s)}$, $\delta:=-A^2-A^{-2}$.

Let $s$ be a state of $D$ having a type A splitting, and let $s'$ denote the state of $D$ obtained from $s$ by replacing the type A splitting
with a type B splitting. Then we have
$$\langle D|s'\rangle=\langle D|s\rangle A^{-2},\ \ 
\mu(s')\leq \mu(s)+1,\ \ \mu(s)\leq \mu(s')+1.$$
Hence we have
\begin{align*}
\max\ \deg\ \langle D|s'\rangle\delta^{\mu(s')-1} 
&\leq \max\ \deg\ \langle D|s\rangle\delta^{\mu(s)-1}, \\
\min\ \deg\ \langle D|s'\rangle\delta^{\mu(s')-1} 
&\leq \min\ \deg\ \langle D|s\rangle\delta^{\mu(s)-1}.
\end{align*}
Let $s_A$ (resp. $s_B$) denote the state of $D$ whose splitting at each crossing is of type A (resp. of type B).
Then we have
\begin{align*}
\max\ \deg\ \langle D\rangle
&\leq \max\ \deg\ \langle D|s_{A}\rangle\delta^{\mu(s_A)-1}
=d(D)+2(\mu(s_A)-1), \\
\min\ \deg\ \langle D\rangle
&\geq \min\ \deg\ \langle D|s_{B}\rangle\delta^{\mu(s_B)-1}
=-d(D)-2(\mu(s_B)-1).
\end{align*}
From these inequalities, we have
$$\spn \langle D\rangle
\leq 2d(D)+2(\mu(s_A)+\mu(s_B)-2).$$

\begin{lem}
\label{lem:d+2}
We have $\mu(s_A)+\mu(s_B)\leq d(D)+2$.
\end{lem}

\begin{proof}
If $d(D)=0$, the inequality is obvious.
Let $d(D)>0$ and
choose a crossing of $D$ and consider the two splittings of it as shown in Figure \ref{split}.
Then, at least one of them is connected and satisfies condition (\ref{Xform}) by virtue of the assumption (\ref{Xform}) on $D$.
Let $D'$ be such a generalized link diagram and assume that
$D'$ is obtained from the type A splitting (the other case is treated similarly).
Let $s_A'$ and $s_B'$ be the states of $D'$ defined by the same way as we introduce $s_A$ and $s_B$ to $D$.
Then $\mu(s_A)=\mu(s_A')$ and $\mu(s_B)\le \mu(s_B')+1$,
hence $\mu(s_A)+\mu(s_B)\le \mu(s'_A)+\mu(s'_B)+1$.
Then the assertion is proved by induction on $d(D)$.
\end{proof}

By Lemma \ref{lem:d+2} we conclude 
$$\spn \langle D\rangle
\leq 2d(D)+2(\mu(s_A)+\mu(s_B)-2)\leq 4d(D).$$
This completes the proof of Proposition \ref{prop4.1}.
\end{proof}

\section{Chord diagrammatic description}
For a curve $\gamma$ on $S$, the bracket polynomial $\langle D_{\gamma} \rangle$ is actually determined by a regular neighborhood of $\gamma(I)$ in $S$.
In this section, we study $\langle D_{\gamma} \rangle$ from this point of view.

Let $d$ be a positive integer.
An \emph{oriented linear chord diagram} of $d$ chords
is a set $C=\{ (i_1,j_1),\ldots,(i_d,j_d)\}$ of $d$ ordered pairs of
integers such that $\{ i_k\}_k \cup \{ j_k\}_k=\{1,\ldots,2d\}$.
Each element of $C$ is called a \emph{chord} of $C$.
A chord $(i,j)$ is called \emph{positive} if $i<j$, and \emph{negative} otherwise.
Finally, a \emph{state} of $C$ is a map $s\colon C\to \{A,B\}$,
where $A$ and $B$ are fixed symbols.

Let $\gamma$ be a curve with $d(\gamma)=d$.
Then the inverse image of the double points of $\gamma$ are $2d$ points on $I$.
We name them $\{ p_i\}_i$ so that $0<p_1<p_2<\cdots <p_{2d}<1$.
The oriented linear chord diagram $C_{\gamma}$ is defined by the condition
that an ordered pair $(i,j)$ is a chord of $C_{\gamma}$ if and only if
$\gamma(p_i)=\gamma(p_j)$ and the pair
$(d\gamma/dt(p_i),d\gamma/dt(p_j))$ of tangent vectors
matches the orientation of $S$.

\begin{rem}
Conversely,
for any oriented linear chord diagram $C$, there is a curve $\gamma$ on some oriented surface $S$ such that $C=C_{\gamma}$.
\end{rem}

Let $C$ be an oriented linear chord diagram of $d$ chords
and $s$ a state of $C$.
For each chord $c=(i,j)\in C$, we define a subset $R_c\subset \mathfrak{S}_{2d+1}$ of permutations of $2d+1$ letters $\{ 0,1,\ldots,2d\}$ in the following way.
\begin{itemize}
\item
If $s(c)=A$ and $c$ is positive, or $s(c)=B$ and $c$ is negative,
then we set $R_c=\{ (i,j-1),(i-1,j)\}$.
\item
If $s(c)=A$ and $c$ is negative, or $s(c)=B$ and $c$ is positive,
then we set $R_c=\{ (i,j),(i-1,j-1)\}$.
\end{itemize}
Consider the subgroup of $\mathfrak{S}_{2d+1}$ generated by $\bigcup_{c\in C} R_c$, and let $\Gamma_s$ be the number of orbits of the action of this group on $\{ 0,\ldots,2d\}$.

We set
\[
\langle C|s \rangle :=
A^{|s^{-1}(A)|-|s^{-1}(B)|} (-A^2-A^{-2})^{\Gamma_s-1},
\]
where $|s^{-1}(A)|$ denotes the cardinality of the set $s^{-1}(A)$,
and we define
\[
\langle C \rangle:=\sum_s \langle C|s \rangle,
\]
where the sum runs over all states of $C$.
We also define
\[
\LL_C(A):=(-A)^{-3w(C)} \langle C \rangle,
\]
where $w(C)$ is the number of positive chords minus the number of negative chords of $C$.

\begin{prop}
\label{prop:g-C}
Let $\gamma$ be a curve on $S$.
Then $\langle D_{\gamma} \rangle=\langle C_{\gamma} \rangle$
and $\LL_{\gamma}(A)=\LL_{C_{\gamma}}(A)$.
\end{prop}

\begin{proof}
First of all, the second formula follows from the first, since
$w(D_{\gamma})=w(C_{\gamma})$.

Now introduce $2d+1$ points $q_i$, $0\le i\le 2d$.
With respect to the parametrization of $\gamma$, these points have the following interpretation:
$q_0=0$, $q_i=(p_i+p_{i+1})/2$ for $1\le i\le 2d-1$, and $q_{2d}=1$.
For a state $s$ of $\gamma$, let $\Gamma(C_{\gamma},s)$ be the graph with the set of vertices being $\{ q_i\}_i$, and the set of edges determined by the condition that $q_k$ and $q_l$ are connected by an edge if and only if $(k,l)\in \bigcup_{c\in C_{\gamma}}R_c$. 
Then $\Gamma(C_{\gamma},s)$ is homeomorphic to the splice of $D_{\gamma}$ by $s$.
See Figure \ref{fig:4cases}.
(Here and in what follows,
we assume the counter-clockwise orientation in any figure.)
The first formula follows from this observation.
\end{proof}

\begin{center}
\begin{figure}
\input{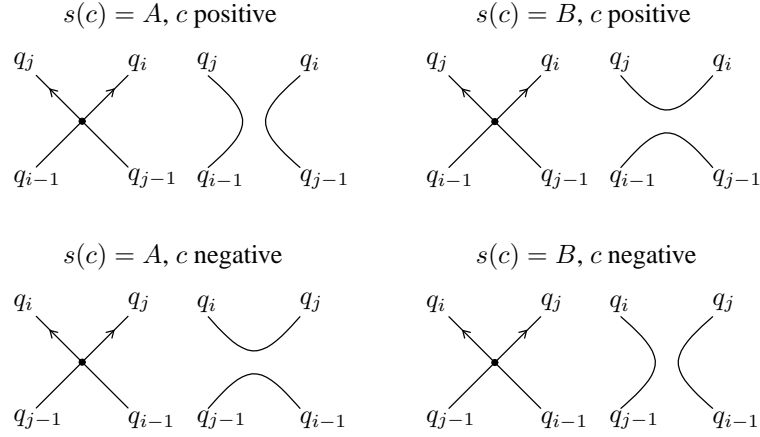}
\caption{proof of Proposition \ref{prop:g-C}}
\label{fig:4cases}
\end{figure}
\end{center}

In the below, we record elementary properties of $\langle D_{\gamma} \rangle$ in terms of chord diagrams.

Let $C=\{ (i_1,j_1),\ldots,(i_d,j_d)\}$ be an oriented linear chord diagram.
Fix $0\le \ell \le 2d$.
For $i\in \{ 1,\ldots,2d\}$, we set
\[
i':=\begin{cases}
i & \quad \text{if $i\le \ell$}, \\
i+2 & \quad \text{if $i>\ell$}.
\end{cases}
\]
We define
\begin{align*}
C_+^\ell&:=\{(i'_k,j'_k)\}_k \cup \{ (\ell,\ell+1)\}, \\
C_-^\ell&:=\{(i'_k,j'_k)\}_k \cup \{ (\ell+1,\ell)\}.
\end{align*}
Also, we define
\begin{align*}
C_+^{\wedge}&:=\{(i_k+1,j_k+1)\}_k \cup \{ (1,2d+2)\}, \\
C_-^{\wedge}&:=\{(i_k+1,j_k+1)\}_k \cup \{ (2d+2,1)\}.
\end{align*}

\begin{prop}[Birth/death of monogons]
\label{prop:monogon}
We have
\begin{align*}
& \langle C_+^\ell \rangle=\langle C_+^{\wedge} \rangle
=(-A^3)\langle C \rangle, \\
& \langle C_-^\ell \rangle=\langle C_-^{\wedge} \rangle
=(-A^{-3})\langle C \rangle.
\end{align*}
\end{prop}

\begin{proof}
If $C=C_{\gamma}$ for some curve $\gamma$, then $C_+^{\ell}$ corresponds to a suitable insertion of a negative monogon to $\gamma$.
Therefore, from the behavior of the bracket polynomial under the Reidemeister move $R_1$, we obtain $\langle C_+^{\ell} \rangle=(-A^3)\langle C \rangle$.
The other cases are treated similarly.
\end{proof}

Let
\[
C=\{ (i_1,j_1),\ldots,(i_d,j_d)\} \quad \text{and} \quad 
D=\{ (k_1,\ell_1),\ldots,(k_e,\ell_e)\}
\]
be oriented linear chord diagrams.
We define the \emph{stacking} of $C$ and $D$ by
\[
C\sharp D:=
\{ (i_a,j_a)\}_a \cup \{(k_b+2d,\ell_b+2d)\}_b.
\]

\begin{prop}[Stacking formula]
\label{prop:stacking}
We have $\langle C\sharp D\rangle=\langle C\rangle \langle D \rangle$.
In particular, $\spn \langle C\sharp D\rangle=\spn \langle C\rangle +\spn \langle D\rangle$.
\end{prop}

\begin{proof}
Since the chords of $C\sharp D$ are in one-to-one correspondence with
the disjoint union of the chords of $C$ and $D$,
any state of $C\sharp D$ is of the form $s\sharp t$, where
$s$ is a state of $C$ and $t$ is a state of $D$.
The assertion follows from the observation that
$|\Gamma(C\sharp D,s\sharp t)|=|\Gamma(C,s)|+|\Gamma(D,t)|-1$.
\end{proof}

\begin{prop}
\label{prop:parity}
Let $C$ be an oriented linear chord diagram of $d$ chords.
\begin{itemize}
\item If $d$ is even, then $\langle C \rangle$ has
only terms of even degree.
\item If $d$ is odd, then $\langle C \rangle$ has
only terms of odd degree.
\end{itemize}
\end{prop}

\begin{proof}
By definition, $\langle C|s \rangle$ has this property, so does $\langle C \rangle$.
\end{proof}

\begin{prop}[Reversing all the chords]
Let $C=\{ (i_1,j_1),\ldots,(i_d,j_d)\}$ be an oriented linear chord diagram and set $\overline{C}:=\{ (j_1,i_1),\ldots,(j_d,i_d)\}$.
Then $\langle \overline{C}\rangle=\langle C\rangle|_{A\mapsto A^{-1}}$.
\end{prop}

\begin{proof}
There is a natural bijection $\iota$ from the set of chords of $C$ to that of $\overline{C}$ given by $(i_k,j_k)\mapsto (j_k,i_k)$.
This maps positive (resp. negative) chords to negative (resp. positive) chords.
Moreover, it induces a bijection from the set of states of $C$ to that of $\overline{C}$ given by $s\mapsto \overline{s}$, determined by the condition that $\{s(c),\overline{s}(\iota(c))\}=\{A,B\}$ for any chord $c$ of $C$.
Then, it holds that $\langle C|s\rangle=\langle \overline{C}| \overline{s} \rangle|_{A\mapsto A^{-1}}$ for any state $s$ of $C$.
This proves the formula.
\end{proof}

\section{The range of the span}

In this section, we study the range of $\spn \langle C\rangle$.
By Theorem \ref{thm1.5}, $\spn \langle C\rangle \le 4d$ if $C$ has $d$ chords.
Also, by Proposition \ref{prop:parity}, $\spn \langle C\rangle$ is always an even integer.
Fixing $d$, let us consider which even integers not greater than $4d$ are realized as $\spn \langle C\rangle$ for some $C$ with $d$ chords.

We say that an even integer $l$ is \emph{$d$-realizable} if there exists an oriented linear chord diagram $C$ of $d$ chords such that $\spn \langle C\rangle=l$.

If $d=1$, $C=C_1:=\{(1,2)\}$ or $C=\overline{C_1}$.
Thus $\langle C\rangle=-A^{\pm 3}$ and $\spn \langle C\rangle=0$.

If $d=2$, by a direct computation, we see that $0$ and $6$ are $2$-realizable, while $2$, $4$, and $8$ are not.
For example, $C_2=\{(1,3),(2,4)\}$ satisfies $\langle C_2\rangle=A^2+1-A^{-4}$, so that $\spn \langle C_2\rangle=6$.

If $d=3$, we see that $0$, $6$, $10$, and $12$ are $3$-realizable, while $2$, $4$, and $8$ are not.
For example, the stacking $C_1\sharp C_2$ satisfies $\spn \langle C_1\sharp C_2 \rangle=6$;
$C_3=\{ (1,5),(2,4),(6,3)\}$ satisfies $\langle C_3 \rangle=-A^5-A^3+A+A^{-1}-A^{-5}$, so that $\spn \langle C_3 \rangle=10$;
the chord diagram $C(3)$ in Example \ref{ex:odd} below satisfies $\spn \langle C(3) \rangle=12$.

To see the case $d\ge 4$, we consider the following two examples.

\begin{example}
\label{ex:odd}
Let $d\ge 1$ be an odd integer, and set
$
C(d):=\{ (i,i+d)\}_{i=1}^d.
$
Then
\[
\langle C(d) \rangle=\sum_{i=1}^{d-1} (-1)^{i-1} A^{-3d-2+4i}-A^{d+2}.
\]
In particular, if $d\ge 3$, then $\spn \langle C \rangle=(d+2)-(-3d+2)=4d$.
\end{example}

\begin{center}
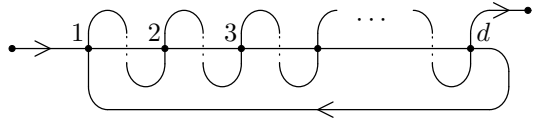
\begin{figure}
{\unitlength 0.1in%
\begin{picture}(27.0000,6.3200)(2.0000,-9.6000)%
%
\special{pn 8}%
\special{ar 700 520 100 120 3.1415927 6.2831853}%
%
\special{pn 8}%
\special{ar 900 680 100 120 6.2831853 3.1415927}%
%
\special{pn 8}%
\special{ar 1100 520 100 120 3.1415927 6.2831853}%
%
\special{pn 8}%
\special{ar 1500 520 100 120 3.1415927 6.2831853}%
%
\special{pn 8}%
\special{ar 1300 680 100 120 6.2831853 3.1415927}%
%
\special{pn 8}%
\special{ar 1700 680 100 120 6.2831853 3.1415927}%
%
\special{pn 8}%
\special{ar 1900 520 100 120 3.1415927 4.7123890}%
%
\special{pn 8}%
\special{ar 2300 520 100 120 4.7123890 6.2831853}%
%
\special{pn 8}%
\special{ar 2500 680 100 120 6.2831853 3.1415927}%
%
\special{pn 8}%
\special{ar 2700 520 100 120 3.1415927 4.7123890}%
%
\special{pn 8}%
\special{pa 2700 400}%
\special{pa 2900 400}%
\special{fp}%
%
\special{pn 4}%
\special{sh 1}%
\special{ar 200 600 16 16 0 6.2831853}%
%
\special{pn 8}%
\special{ar 700 800 100 120 1.5707963 3.1415927}%
%
\special{pn 8}%
\special{ar 2700 800 100 120 6.2831853 1.5707963}%
%
\special{pn 8}%
\special{ar 2700 720 100 120 4.7123890 6.2831853}%
%
\special{pn 4}%
\special{sh 1}%
\special{ar 2900 400 16 16 0 6.2831853}%
%
\special{pn 8}%
\special{pa 200 600}%
\special{pa 2700 600}%
\special{fp}%
%
\special{pn 8}%
\special{pa 1000 680}%
\special{pa 1000 520}%
\special{fp}%
%
\special{pn 8}%
\special{pa 1400 680}%
\special{pa 1400 520}%
\special{fp}%
%
\special{pn 8}%
\special{pa 1800 680}%
\special{pa 1800 520}%
\special{fp}%
%
\special{pn 8}%
\special{pa 2600 520}%
\special{pa 2600 680}%
\special{fp}%
%
\special{pn 4}%
\special{sh 1}%
\special{ar 2600 600 16 16 0 6.2831853}%
%
\special{pn 4}%
\special{sh 1}%
\special{ar 1800 600 16 16 0 6.2831853}%
%
\special{pn 4}%
\special{sh 1}%
\special{ar 1400 600 16 16 0 6.2831853}%
%
\special{pn 4}%
\special{sh 1}%
\special{ar 1000 600 16 16 0 6.2831853}%
%
\special{pn 4}%
\special{sh 1}%
\special{ar 600 600 16 16 0 6.2831853}%
\put(20.0000,-4.8000){\makebox(0,0)[lb]{$\cdots$}}%
%
\special{pn 8}%
\special{pa 400 600}%
\special{pa 320 560}%
\special{fp}%
\special{pa 400 600}%
\special{pa 320 640}%
\special{fp}%
%
\special{pn 8}%
\special{pa 2800 400}%
\special{pa 2720 360}%
\special{fp}%
\special{pa 2800 400}%
\special{pa 2720 440}%
\special{fp}%
\put(5.1000,-5.6000){\makebox(0,0)[lb]{$1$}}%
\put(9.1000,-5.6000){\makebox(0,0)[lb]{$2$}}%
\put(26.3000,-5.6000){\makebox(0,0)[lb]{$d$}}%
\put(13.1000,-5.6000){\makebox(0,0)[lb]{$3$}}%
%
\special{pn 8}%
\special{pa 800 680}%
\special{pa 800 520}%
\special{dt 0.045}%
%
\special{pn 8}%
\special{pa 1200 680}%
\special{pa 1200 520}%
\special{dt 0.045}%
%
\special{pn 8}%
\special{pa 1600 680}%
\special{pa 1600 520}%
\special{dt 0.045}%
%
\special{pn 8}%
\special{pa 2400 680}%
\special{pa 2400 520}%
\special{dt 0.045}%
%
\special{pn 8}%
\special{pa 2700 920}%
\special{pa 700 920}%
\special{fp}%
%
\special{pn 8}%
\special{pa 1800 920}%
\special{pa 1880 880}%
\special{fp}%
\special{pa 1800 920}%
\special{pa 1880 960}%
\special{fp}%
%
\special{pn 8}%
\special{pa 600 800}%
\special{pa 600 520}%
\special{fp}%
%
\special{pn 8}%
\special{pa 2800 800}%
\special{pa 2800 720}%
\special{fp}%
\end{picture}}%
\caption{the curve $\gamma_d$ in Example \ref{ex:odd}}
\label{fig:gamma_odd}
\end{figure}
\end{center}

\begin{proof}
We have $C(d)=C_{\gamma_d}$, where $\gamma_d$ is the curve as shown in Figure \ref{fig:gamma_odd}.
Let $d\ge 1$ be an odd integer. Then
\begin{align*}
& \quad 
{\unitlength 0.1in%
\begin{picture}(37.0000,6.3200)(2.0000,-9.6000)%
%
\special{pn 8}%
\special{ar 1500 520 100 120 3.1415927 6.2831853}%
%
\special{pn 8}%
\special{ar 1700 680 100 120 6.2831853 3.1415927}%
%
\special{pn 8}%
\special{ar 1900 520 100 120 3.1415927 6.2831853}%
%
\special{pn 8}%
\special{ar 2300 520 100 120 3.1415927 6.2831853}%
%
\special{pn 8}%
\special{ar 2100 680 100 120 6.2831853 3.1415927}%
%
\special{pn 8}%
\special{ar 2500 680 100 120 6.2831853 3.1415927}%
%
\special{pn 8}%
\special{ar 2700 520 100 120 3.1415927 4.7123890}%
%
\special{pn 8}%
\special{ar 3100 520 100 120 4.7123890 6.2831853}%
%
\special{pn 8}%
\special{ar 3300 680 100 120 6.2831853 3.1415927}%
%
\special{pn 8}%
\special{ar 3500 520 100 120 3.1415927 4.7123890}%
%
\special{pn 8}%
\special{pa 3500 400}%
\special{pa 3700 400}%
\special{fp}%
%
\special{pn 4}%
\special{sh 1}%
\special{ar 1000 600 16 16 0 6.2831853}%
%
\special{pn 8}%
\special{ar 1500 800 100 120 1.5707963 3.1415927}%
%
\special{pn 8}%
\special{ar 3500 800 100 120 6.2831853 1.5707963}%
%
\special{pn 8}%
\special{ar 3500 720 100 120 4.7123890 6.2831853}%
%
\special{pn 8}%
\special{pa 3500 920}%
\special{pa 1500 920}%
\special{fp}%
%
\special{pn 4}%
\special{sh 1}%
\special{ar 3700 400 16 16 0 6.2831853}%
%
\special{pn 8}%
\special{pa 1000 600}%
\special{pa 3500 600}%
\special{fp}%
%
\special{pn 8}%
\special{pa 1800 680}%
\special{pa 1800 520}%
\special{fp}%
%
\special{pn 8}%
\special{pa 2200 680}%
\special{pa 2200 520}%
\special{fp}%
%
\special{pn 8}%
\special{pa 2600 680}%
\special{pa 2600 520}%
\special{fp}%
%
\special{pn 8}%
\special{pa 3400 520}%
\special{pa 3400 680}%
\special{fp}%
%
\special{pn 4}%
\special{sh 1}%
\special{ar 3400 600 16 16 0 6.2831853}%
%
\special{pn 4}%
\special{sh 1}%
\special{ar 2600 600 16 16 0 6.2831853}%
%
\special{pn 4}%
\special{sh 1}%
\special{ar 2200 600 16 16 0 6.2831853}%
%
\special{pn 4}%
\special{sh 1}%
\special{ar 1800 600 16 16 0 6.2831853}%
%
\special{pn 4}%
\special{sh 1}%
\special{ar 1400 600 16 16 0 6.2831853}%
\put(28.0000,-4.8000){\makebox(0,0)[lb]{$\cdots$}}%
%
\special{pn 8}%
\special{pa 1200 600}%
\special{pa 1120 560}%
\special{fp}%
\special{pa 1200 600}%
\special{pa 1120 640}%
\special{fp}%
%
\special{pn 8}%
\special{pa 2600 920}%
\special{pa 2680 880}%
\special{fp}%
\special{pa 2600 920}%
\special{pa 2680 960}%
\special{fp}%
%
\special{pn 8}%
\special{pa 3600 400}%
\special{pa 3520 360}%
\special{fp}%
\special{pa 3600 400}%
\special{pa 3520 440}%
\special{fp}%
\put(13.1000,-5.6000){\makebox(0,0)[lb]{$1$}}%
\put(17.1000,-5.6000){\makebox(0,0)[lb]{$2$}}%
\put(34.3000,-5.6000){\makebox(0,0)[lb]{$d+2$}}%
\put(21.1000,-5.6000){\makebox(0,0)[lb]{$3$}}%
\put(2.0000,-7.3500){\makebox(0,0)[lb]{$\langle \gamma_{d+2} \rangle=$}}%
%
\special{pn 8}%
\special{pa 2400 520}%
\special{pa 2400 680}%
\special{dt 0.045}%
%
\special{pn 8}%
\special{pa 1600 520}%
\special{pa 1600 680}%
\special{dt 0.045}%
%
\special{pn 8}%
\special{pa 2000 520}%
\special{pa 2000 680}%
\special{dt 0.045}%
%
\special{pn 8}%
\special{pa 3200 520}%
\special{pa 3200 680}%
\special{dt 0.045}%
%
\special{pn 8}%
\special{pa 1400 800}%
\special{pa 1400 520}%
\special{fp}%
%
\special{pn 8}%
\special{pa 3600 800}%
\special{pa 3600 720}%
\special{fp}%
%
\special{pn 13}%
\special{pa 900 400}%
\special{pa 800 660}%
\special{fp}%
\special{pa 800 660}%
\special{pa 900 920}%
\special{fp}%
%
\special{pn 13}%
\special{pa 3800 400}%
\special{pa 3900 660}%
\special{fp}%
\special{pa 3900 660}%
\special{pa 3800 920}%
\special{fp}%
\end{picture}}%
 \\
& \hspace{4em} 
{\unitlength 0.1in%
\begin{picture}(34.0000,5.9200)(5.0000,-7.2000)%
%
\special{pn 8}%
\special{ar 1500 320 100 120 3.1415927 6.2831853}%
%
\special{pn 8}%
\special{ar 1700 480 100 120 6.2831853 3.1415927}%
%
\special{pn 8}%
\special{ar 1900 320 100 120 3.1415927 6.2831853}%
%
\special{pn 8}%
\special{ar 2300 320 100 120 3.1415927 6.2831853}%
%
\special{pn 8}%
\special{ar 2100 480 100 120 6.2831853 3.1415927}%
%
\special{pn 8}%
\special{ar 2500 480 100 120 6.2831853 3.1415927}%
%
\special{pn 8}%
\special{ar 2700 320 100 120 3.1415927 4.7123890}%
%
\special{pn 8}%
\special{ar 3100 320 100 120 4.7123890 6.2831853}%
%
\special{pn 8}%
\special{ar 3300 480 100 120 6.2831853 3.1415927}%
%
\special{pn 8}%
\special{ar 3500 320 100 120 3.1415927 4.7123890}%
%
\special{pn 8}%
\special{pa 3500 200}%
\special{pa 3700 200}%
\special{fp}%
%
\special{pn 4}%
\special{sh 1}%
\special{ar 1000 400 16 16 0 6.2831853}%
%
\special{pn 8}%
\special{ar 1500 600 100 120 1.5707963 3.1415927}%
%
\special{pn 8}%
\special{ar 3500 600 100 120 6.2831853 1.5707963}%
%
\special{pn 8}%
\special{ar 3500 520 100 120 4.7123890 6.2831853}%
%
\special{pn 8}%
\special{pa 3500 720}%
\special{pa 1500 720}%
\special{fp}%
%
\special{pn 4}%
\special{sh 1}%
\special{ar 3700 200 16 16 0 6.2831853}%
%
\special{pn 4}%
\special{sh 1}%
\special{ar 3400 400 16 16 0 6.2831853}%
%
\special{pn 4}%
\special{sh 1}%
\special{ar 2600 400 16 16 0 6.2831853}%
%
\special{pn 4}%
\special{sh 1}%
\special{ar 2200 400 16 16 0 6.2831853}%
%
\special{pn 4}%
\special{sh 1}%
\special{ar 1800 400 16 16 0 6.2831853}%
\put(28.0000,-2.8000){\makebox(0,0)[lb]{$\cdots$}}%
\put(17.1000,-3.6000){\makebox(0,0)[lb]{$2$}}%
\put(34.3000,-3.6000){\makebox(0,0)[lb]{$d+2$}}%
\put(21.1000,-3.6000){\makebox(0,0)[lb]{$3$}}%
%
\special{pn 8}%
\special{ar 1320 320 80 80 6.2831853 1.5707963}%
%
\special{pn 8}%
\special{ar 1480 480 80 80 3.1415927 4.7123890}%
%
\special{pn 8}%
\special{pa 1320 400}%
\special{pa 1000 400}%
\special{fp}%
%
\special{pn 8}%
\special{pa 1480 400}%
\special{pa 3500 400}%
\special{fp}%
\put(5.0000,-5.1000){\makebox(0,0)[lb]{$=A$}}%
%
\special{pn 8}%
\special{pa 1800 480}%
\special{pa 1800 440}%
\special{fp}%
%
\special{pn 8}%
\special{pa 1800 360}%
\special{pa 1800 320}%
\special{fp}%
%
\special{pn 8}%
\special{pa 2200 320}%
\special{pa 2200 360}%
\special{fp}%
%
\special{pn 8}%
\special{pa 2200 440}%
\special{pa 2200 480}%
\special{fp}%
%
\special{pn 8}%
\special{pa 2600 480}%
\special{pa 2600 440}%
\special{fp}%
%
\special{pn 8}%
\special{pa 2600 360}%
\special{pa 2600 320}%
\special{fp}%
%
\special{pn 8}%
\special{pa 3400 480}%
\special{pa 3400 440}%
\special{fp}%
%
\special{pn 8}%
\special{pa 3400 360}%
\special{pa 3400 320}%
\special{fp}%
%
\special{pn 8}%
\special{pa 1600 320}%
\special{pa 1600 480}%
\special{dt 0.045}%
%
\special{pn 8}%
\special{pa 2000 320}%
\special{pa 2000 480}%
\special{dt 0.045}%
%
\special{pn 8}%
\special{pa 2400 320}%
\special{pa 2400 480}%
\special{dt 0.045}%
%
\special{pn 8}%
\special{pa 3200 320}%
\special{pa 3200 480}%
\special{dt 0.045}%
%
\special{pn 13}%
\special{pa 900 200}%
\special{pa 800 460}%
\special{fp}%
\special{pa 800 460}%
\special{pa 900 720}%
\special{fp}%
%
\special{pn 8}%
\special{pa 1400 600}%
\special{pa 1400 480}%
\special{fp}%
%
\special{pn 8}%
\special{pa 3600 600}%
\special{pa 3600 520}%
\special{fp}%
%
\special{pn 13}%
\special{pa 3800 200}%
\special{pa 3900 460}%
\special{fp}%
\special{pa 3900 460}%
\special{pa 3800 720}%
\special{fp}%
\end{picture}}%
 \\
& \hspace{6em} 
{\unitlength 0.1in%
\begin{picture}(35.0000,5.9200)(2.0000,-7.2000)%
%
\special{pn 8}%
\special{ar 1300 320 100 120 3.1415927 6.2831853}%
%
\special{pn 8}%
\special{ar 1500 480 100 120 6.2831853 3.1415927}%
%
\special{pn 8}%
\special{ar 1700 320 100 120 3.1415927 6.2831853}%
%
\special{pn 8}%
\special{ar 2100 320 100 120 3.1415927 6.2831853}%
%
\special{pn 8}%
\special{ar 1900 480 100 120 6.2831853 3.1415927}%
%
\special{pn 8}%
\special{ar 2300 480 100 120 6.2831853 3.1415927}%
%
\special{pn 8}%
\special{ar 2500 320 100 120 3.1415927 4.7123890}%
%
\special{pn 8}%
\special{ar 2900 320 100 120 4.7123890 6.2831853}%
%
\special{pn 8}%
\special{ar 3100 480 100 120 6.2831853 3.1415927}%
%
\special{pn 8}%
\special{ar 3300 320 100 120 3.1415927 4.7123890}%
%
\special{pn 8}%
\special{pa 3300 200}%
\special{pa 3500 200}%
\special{fp}%
%
\special{pn 4}%
\special{sh 1}%
\special{ar 800 400 16 16 0 6.2831853}%
%
\special{pn 8}%
\special{ar 1300 600 100 120 1.5707963 3.1415927}%
%
\special{pn 8}%
\special{ar 3300 600 100 120 6.2831853 1.5707963}%
%
\special{pn 8}%
\special{ar 3300 520 100 120 4.7123890 6.2831853}%
%
\special{pn 8}%
\special{pa 3300 720}%
\special{pa 1300 720}%
\special{fp}%
%
\special{pn 4}%
\special{sh 1}%
\special{ar 3500 200 16 16 0 6.2831853}%
%
\special{pn 4}%
\special{sh 1}%
\special{ar 3200 400 16 16 0 6.2831853}%
%
\special{pn 4}%
\special{sh 1}%
\special{ar 2400 400 16 16 0 6.2831853}%
%
\special{pn 4}%
\special{sh 1}%
\special{ar 2000 400 16 16 0 6.2831853}%
%
\special{pn 4}%
\special{sh 1}%
\special{ar 1600 400 16 16 0 6.2831853}%
\put(26.0000,-2.8000){\makebox(0,0)[lb]{$\cdots$}}%
\put(15.1000,-3.6000){\makebox(0,0)[lb]{$2$}}%
\put(32.3000,-3.6000){\makebox(0,0)[lb]{$d+2$}}%
\put(19.1000,-3.6000){\makebox(0,0)[lb]{$3$}}%
%
\special{pn 8}%
\special{ar 1120 480 80 80 4.7123890 6.2831853}%
%
\special{pn 8}%
\special{ar 1280 320 80 80 1.5707963 3.1415927}%
%
\special{pn 8}%
\special{pa 1120 400}%
\special{pa 800 400}%
\special{fp}%
%
\special{pn 8}%
\special{pa 1280 400}%
\special{pa 3300 400}%
\special{fp}%
\put(2.0000,-5.1000){\makebox(0,0)[lb]{$+A^{-1}$}}%
%
\special{pn 8}%
\special{pa 1600 480}%
\special{pa 1600 440}%
\special{fp}%
%
\special{pn 8}%
\special{pa 1600 360}%
\special{pa 1600 320}%
\special{fp}%
%
\special{pn 8}%
\special{pa 2000 320}%
\special{pa 2000 360}%
\special{fp}%
%
\special{pn 8}%
\special{pa 2000 440}%
\special{pa 2000 480}%
\special{fp}%
%
\special{pn 8}%
\special{pa 2400 480}%
\special{pa 2400 440}%
\special{fp}%
%
\special{pn 8}%
\special{pa 2400 360}%
\special{pa 2400 320}%
\special{fp}%
%
\special{pn 8}%
\special{pa 3200 480}%
\special{pa 3200 440}%
\special{fp}%
%
\special{pn 8}%
\special{pa 3200 360}%
\special{pa 3200 320}%
\special{fp}%
%
\special{pn 8}%
\special{pa 1400 320}%
\special{pa 1400 480}%
\special{dt 0.045}%
%
\special{pn 8}%
\special{pa 1800 320}%
\special{pa 1800 480}%
\special{dt 0.045}%
%
\special{pn 8}%
\special{pa 2200 320}%
\special{pa 2200 480}%
\special{dt 0.045}%
%
\special{pn 8}%
\special{pa 3000 320}%
\special{pa 3000 480}%
\special{dt 0.045}%
%
\special{pn 13}%
\special{pa 700 200}%
\special{pa 600 460}%
\special{fp}%
\special{pa 600 460}%
\special{pa 700 720}%
\special{fp}%
%
\special{pn 8}%
\special{pa 1200 600}%
\special{pa 1200 480}%
\special{fp}%
%
\special{pn 8}%
\special{pa 3400 600}%
\special{pa 3400 520}%
\special{fp}%
%
\special{pn 13}%
\special{pa 3600 200}%
\special{pa 3700 460}%
\special{fp}%
\special{pa 3700 460}%
\special{pa 3600 720}%
\special{fp}%
\end{picture}}%
 \\
& \hspace{4em} 
{\unitlength 0.1in%
\begin{picture}(35.0000,5.9200)(2.0000,-7.2000)%
%
\special{pn 8}%
\special{ar 1300 320 100 120 3.1415927 6.2831853}%
%
\special{pn 8}%
\special{ar 1500 480 100 120 6.2831853 3.1415927}%
%
\special{pn 8}%
\special{ar 1700 320 100 120 3.1415927 6.2831853}%
%
\special{pn 8}%
\special{ar 2100 320 100 120 3.1415927 6.2831853}%
%
\special{pn 8}%
\special{ar 1900 480 100 120 6.2831853 3.1415927}%
%
\special{pn 8}%
\special{ar 2300 480 100 120 6.2831853 3.1415927}%
%
\special{pn 8}%
\special{ar 2500 320 100 120 3.1415927 4.7123890}%
%
\special{pn 8}%
\special{ar 2900 320 100 120 4.7123890 6.2831853}%
%
\special{pn 8}%
\special{ar 3100 480 100 120 6.2831853 3.1415927}%
%
\special{pn 8}%
\special{ar 3300 320 100 120 3.1415927 4.7123890}%
%
\special{pn 8}%
\special{pa 3300 200}%
\special{pa 3500 200}%
\special{fp}%
%
\special{pn 4}%
\special{sh 1}%
\special{ar 800 400 16 16 0 6.2831853}%
%
\special{pn 8}%
\special{ar 1300 600 100 120 1.5707963 3.1415927}%
%
\special{pn 8}%
\special{ar 3300 600 100 120 6.2831853 1.5707963}%
%
\special{pn 8}%
\special{ar 3300 520 100 120 4.7123890 6.2831853}%
%
\special{pn 8}%
\special{pa 3300 720}%
\special{pa 1300 720}%
\special{fp}%
%
\special{pn 4}%
\special{sh 1}%
\special{ar 3500 200 16 16 0 6.2831853}%
%
\special{pn 4}%
\special{sh 1}%
\special{ar 3200 400 16 16 0 6.2831853}%
%
\special{pn 4}%
\special{sh 1}%
\special{ar 2400 400 16 16 0 6.2831853}%
%
\special{pn 4}%
\special{sh 1}%
\special{ar 2000 400 16 16 0 6.2831853}%
\put(26.0000,-2.8000){\makebox(0,0)[lb]{$\cdots$}}%
\put(32.3000,-3.6000){\makebox(0,0)[lb]{$d+2$}}%
\put(19.1000,-3.6000){\makebox(0,0)[lb]{$3$}}%
%
\special{pn 8}%
\special{ar 1120 320 80 80 6.2831853 1.5707963}%
%
\special{pn 8}%
\special{ar 1280 480 80 80 3.1415927 4.7123890}%
%
\special{pn 8}%
\special{pa 1120 400}%
\special{pa 800 400}%
\special{fp}%
%
\special{pn 8}%
\special{ar 1520 320 80 80 6.2831853 1.5707963}%
%
\special{pn 8}%
\special{ar 1680 480 80 80 3.1415927 4.7123890}%
%
\special{pn 8}%
\special{pa 1520 400}%
\special{pa 1280 400}%
\special{fp}%
%
\special{pn 8}%
\special{pa 1680 400}%
\special{pa 3300 400}%
\special{fp}%
\put(2.0000,-5.1000){\makebox(0,0)[lb]{$=A^2$}}%
%
\special{pn 8}%
\special{pa 2000 480}%
\special{pa 2000 440}%
\special{fp}%
%
\special{pn 8}%
\special{pa 2000 360}%
\special{pa 2000 320}%
\special{fp}%
%
\special{pn 8}%
\special{pa 2400 480}%
\special{pa 2400 440}%
\special{fp}%
%
\special{pn 8}%
\special{pa 2400 360}%
\special{pa 2400 320}%
\special{fp}%
%
\special{pn 8}%
\special{pa 3200 480}%
\special{pa 3200 440}%
\special{fp}%
%
\special{pn 8}%
\special{pa 3200 360}%
\special{pa 3200 320}%
\special{fp}%
%
\special{pn 8}%
\special{pa 1400 320}%
\special{pa 1400 480}%
\special{dt 0.045}%
%
\special{pn 8}%
\special{pa 1800 320}%
\special{pa 1800 480}%
\special{dt 0.045}%
%
\special{pn 8}%
\special{pa 2200 320}%
\special{pa 2200 480}%
\special{dt 0.045}%
%
\special{pn 8}%
\special{pa 3000 320}%
\special{pa 3000 480}%
\special{dt 0.045}%
%
\special{pn 13}%
\special{pa 700 200}%
\special{pa 600 460}%
\special{fp}%
\special{pa 600 460}%
\special{pa 700 720}%
\special{fp}%
%
\special{pn 8}%
\special{pa 1200 600}%
\special{pa 1200 480}%
\special{fp}%
%
\special{pn 8}%
\special{pa 3400 600}%
\special{pa 3400 520}%
\special{fp}%
%
\special{pn 13}%
\special{pa 3600 200}%
\special{pa 3700 460}%
\special{fp}%
\special{pa 3700 460}%
\special{pa 3600 720}%
\special{fp}%
\end{picture}}%
 \\
& \hspace{6em} 
{\unitlength 0.1in%
\begin{picture}(33.0000,5.9200)(4.0000,-9.2000)%
%
\special{pn 8}%
\special{ar 1300 520 100 120 3.1415927 6.2831853}%
%
\special{pn 8}%
\special{ar 1500 680 100 120 6.2831853 3.1415927}%
%
\special{pn 8}%
\special{ar 1700 520 100 120 3.1415927 6.2831853}%
%
\special{pn 8}%
\special{ar 2100 520 100 120 3.1415927 6.2831853}%
%
\special{pn 8}%
\special{ar 1900 680 100 120 6.2831853 3.1415927}%
%
\special{pn 8}%
\special{ar 2300 680 100 120 6.2831853 3.1415927}%
%
\special{pn 8}%
\special{ar 2500 520 100 120 3.1415927 4.7123890}%
%
\special{pn 8}%
\special{ar 2900 520 100 120 4.7123890 6.2831853}%
%
\special{pn 8}%
\special{ar 3100 680 100 120 6.2831853 3.1415927}%
%
\special{pn 8}%
\special{ar 3300 520 100 120 3.1415927 4.7123890}%
%
\special{pn 8}%
\special{pa 3300 400}%
\special{pa 3500 400}%
\special{fp}%
%
\special{pn 4}%
\special{sh 1}%
\special{ar 800 600 16 16 0 6.2831853}%
%
\special{pn 8}%
\special{ar 1300 800 100 120 1.5707963 3.1415927}%
%
\special{pn 8}%
\special{ar 3300 800 100 120 6.2831853 1.5707963}%
%
\special{pn 8}%
\special{ar 3300 720 100 120 4.7123890 6.2831853}%
%
\special{pn 8}%
\special{pa 3300 920}%
\special{pa 1300 920}%
\special{fp}%
%
\special{pn 4}%
\special{sh 1}%
\special{ar 3500 400 16 16 0 6.2831853}%
%
\special{pn 4}%
\special{sh 1}%
\special{ar 3200 600 16 16 0 6.2831853}%
%
\special{pn 4}%
\special{sh 1}%
\special{ar 2400 600 16 16 0 6.2831853}%
%
\special{pn 4}%
\special{sh 1}%
\special{ar 2000 600 16 16 0 6.2831853}%
\put(26.0000,-4.8000){\makebox(0,0)[lb]{$\cdots$}}%
\put(32.3000,-5.6000){\makebox(0,0)[lb]{$d+2$}}%
\put(19.1000,-5.6000){\makebox(0,0)[lb]{$3$}}%
%
\special{pn 8}%
\special{ar 1120 520 80 80 6.2831853 1.5707963}%
%
\special{pn 8}%
\special{ar 1280 680 80 80 3.1415927 4.7123890}%
%
\special{pn 8}%
\special{pa 1120 600}%
\special{pa 800 600}%
\special{fp}%
%
\special{pn 8}%
\special{ar 1520 680 80 80 4.7123890 6.2831853}%
%
\special{pn 8}%
\special{ar 1680 520 80 80 1.5707963 3.1415927}%
%
\special{pn 8}%
\special{pa 1520 600}%
\special{pa 1280 600}%
\special{fp}%
%
\special{pn 8}%
\special{pa 1680 600}%
\special{pa 3300 600}%
\special{fp}%
\put(4.0000,-7.1000){\makebox(0,0)[lb]{$+$}}%
%
\special{pn 8}%
\special{pa 2000 680}%
\special{pa 2000 640}%
\special{fp}%
%
\special{pn 8}%
\special{pa 2000 560}%
\special{pa 2000 520}%
\special{fp}%
%
\special{pn 8}%
\special{pa 2400 680}%
\special{pa 2400 640}%
\special{fp}%
%
\special{pn 8}%
\special{pa 2400 560}%
\special{pa 2400 520}%
\special{fp}%
%
\special{pn 8}%
\special{pa 3200 680}%
\special{pa 3200 640}%
\special{fp}%
%
\special{pn 8}%
\special{pa 3200 560}%
\special{pa 3200 520}%
\special{fp}%
%
\special{pn 8}%
\special{pa 1400 520}%
\special{pa 1400 680}%
\special{dt 0.045}%
%
\special{pn 8}%
\special{pa 1800 520}%
\special{pa 1800 680}%
\special{dt 0.045}%
%
\special{pn 8}%
\special{pa 2200 520}%
\special{pa 2200 680}%
\special{dt 0.045}%
%
\special{pn 8}%
\special{pa 3000 520}%
\special{pa 3000 680}%
\special{dt 0.045}%
%
\special{pn 13}%
\special{pa 700 400}%
\special{pa 600 660}%
\special{fp}%
\special{pa 600 660}%
\special{pa 700 920}%
\special{fp}%
%
\special{pn 8}%
\special{pa 1200 800}%
\special{pa 1200 680}%
\special{fp}%
%
\special{pn 8}%
\special{pa 3400 800}%
\special{pa 3400 720}%
\special{fp}%
%
\special{pn 13}%
\special{pa 3600 400}%
\special{pa 3700 660}%
\special{fp}%
\special{pa 3700 660}%
\special{pa 3600 920}%
\special{fp}%
\end{picture}}%
 \\
& \hspace{8em} 
{\unitlength 0.1in%
\begin{picture}(18.6000,2.8000)(1.0000,-8.4000)%
\put(1.0000,-7.7500){\makebox(0,0)[lb]{$+A^{-1}\cdot (-A^{-3})^{d+1}$}}%
%
\special{pn 8}%
\special{pa 1500 700}%
\special{pa 1900 700}%
\special{fp}%
%
\special{pn 4}%
\special{sh 1}%
\special{ar 1900 700 16 16 0 6.2831853}%
%
\special{pn 4}%
\special{sh 1}%
\special{ar 1500 700 16 16 0 6.2831853}%
%
\special{pn 13}%
\special{pa 1500 560}%
\special{pa 1440 700}%
\special{fp}%
\special{pa 1440 700}%
\special{pa 1500 840}%
\special{fp}%
%
\special{pn 13}%
\special{pa 1900 560}%
\special{pa 1960 700}%
\special{fp}%
\special{pa 1960 700}%
\special{pa 1900 840}%
\special{fp}%
\end{picture}}%
 \\
& \hspace{4em} = A^2 \langle \gamma_d \rangle+(-A^{-3})^d+A^{-3d-4}.
\end{align*}

Now it is easy to see that $\langle \gamma_1 \rangle=-A^3$,
and the formula is proved by an inductive argument.
\end{proof}

\begin{example}
\label{ex:even}
Let $d\ge 4$ be an even integer, and set
\[
C(d):=\{(1,d),(d+1,2d)\} \cup \{(2d-i,i+1)\}_{i=1}^{d-2}.
\]
Then
\[
\langle C(d) \rangle=
A^{-3d+4}-A^{-3d+8}
+2\left( \sum_{i=1}^{d-4} (-1)^{i-1} A^{-3d+8+4i} \right)
+A^{d-4}-A^d+A^{d+4}.
\]
In particular, $\spn \langle C(d) \rangle=(d+4)-(-3d+4)=4d$.
\end{example}

\begin{center}
\begin{figure}
{\unitlength 0.1in%
\begin{picture}(29.0000,7.2000)(2.0000,-9.6000)%
%
\special{pn 8}%
\special{ar 900 680 100 120 6.2831853 3.1415927}%
%
\special{pn 8}%
\special{ar 1100 520 100 120 3.1415927 6.2831853}%
%
\special{pn 8}%
\special{ar 1500 520 100 120 3.1415927 6.2831853}%
%
\special{pn 8}%
\special{ar 1300 680 100 120 6.2831853 3.1415927}%
%
\special{pn 8}%
\special{ar 1700 680 100 120 6.2831853 3.1415927}%
%
\special{pn 8}%
\special{ar 2300 520 100 120 3.1415927 4.7123890}%
%
\special{pn 8}%
\special{ar 2620 520 100 120 4.7123890 6.2831853}%
%
\special{pn 8}%
\special{ar 2820 680 100 120 6.2831853 3.1415927}%
%
\special{pn 4}%
\special{sh 1}%
\special{ar 200 600 16 16 0 6.2831853}%
%
\special{pn 8}%
\special{ar 700 800 100 120 1.5707963 3.1415927}%
%
\special{pn 8}%
\special{ar 3000 800 100 120 6.2831853 1.5707963}%
%
\special{pn 8}%
\special{ar 3000 720 100 120 4.7123890 6.2831853}%
%
\special{pn 8}%
\special{pa 1000 680}%
\special{pa 1000 520}%
\special{fp}%
%
\special{pn 8}%
\special{pa 1400 680}%
\special{pa 1400 520}%
\special{fp}%
%
\special{pn 8}%
\special{pa 2200 680}%
\special{pa 2200 520}%
\special{fp}%
%
\special{pn 4}%
\special{sh 1}%
\special{ar 2920 600 16 16 0 6.2831853}%
%
\special{pn 4}%
\special{sh 1}%
\special{ar 1800 600 16 16 0 6.2831853}%
%
\special{pn 4}%
\special{sh 1}%
\special{ar 1400 600 16 16 0 6.2831853}%
%
\special{pn 4}%
\special{sh 1}%
\special{ar 1000 600 16 16 0 6.2831853}%
%
\special{pn 4}%
\special{sh 1}%
\special{ar 600 600 16 16 0 6.2831853}%
\put(23.2000,-4.8000){\makebox(0,0)[lb]{$\cdots$}}%
%
\special{pn 8}%
\special{pa 400 600}%
\special{pa 320 560}%
\special{fp}%
\special{pa 400 600}%
\special{pa 320 640}%
\special{fp}%
%
\special{pn 8}%
\special{pa 1800 920}%
\special{pa 1880 880}%
\special{fp}%
\special{pa 1800 920}%
\special{pa 1880 960}%
\special{fp}%
\put(5.1000,-5.6000){\makebox(0,0)[lb]{$1$}}%
\put(9.1000,-5.6000){\makebox(0,0)[lb]{$2$}}%
\put(13.1000,-5.6000){\makebox(0,0)[lb]{$3$}}%
%
\special{pn 8}%
\special{ar 1900 520 100 120 3.1415927 6.2831853}%
%
\special{pn 8}%
\special{pa 1800 680}%
\special{pa 1800 520}%
\special{fp}%
\put(17.1000,-5.6000){\makebox(0,0)[lb]{$4$}}%
%
\special{pn 8}%
\special{ar 2100 680 100 120 6.2831853 3.1415927}%
%
\special{pn 4}%
\special{sh 1}%
\special{ar 2200 600 16 16 0 6.2831853}%
%
\special{pn 8}%
\special{pa 3000 920}%
\special{pa 700 920}%
\special{fp}%
%
\special{pn 8}%
\special{pa 3000 600}%
\special{pa 200 600}%
\special{fp}%
%
\special{pn 8}%
\special{ar 2820 400 100 120 4.7123890 6.2831853}%
%
\special{pn 8}%
\special{ar 700 400 100 120 3.1415927 4.7123890}%
%
\special{pn 8}%
\special{ar 700 520 100 120 4.7123890 6.2831853}%
%
\special{pn 4}%
\special{sh 1}%
\special{ar 200 400 16 16 0 6.2831853}%
%
\special{pn 8}%
\special{pa 2820 280}%
\special{pa 700 280}%
\special{fp}%
\put(29.5000,-5.6000){\makebox(0,0)[lb]{$d-1$}}%
%
\special{pn 8}%
\special{pa 320 400}%
\special{pa 400 360}%
\special{fp}%
\special{pa 320 400}%
\special{pa 400 440}%
\special{fp}%
%
\special{pn 8}%
\special{pa 1880 280}%
\special{pa 1800 240}%
\special{fp}%
\special{pa 1880 280}%
\special{pa 1800 320}%
\special{fp}%
%
\special{pn 4}%
\special{sh 1}%
\special{ar 600 400 16 16 0 6.2831853}%
%
\special{pn 8}%
\special{pa 700 400}%
\special{pa 200 400}%
\special{fp}%
%
\special{pn 8}%
\special{pa 800 520}%
\special{pa 800 680}%
\special{dt 0.045}%
%
\special{pn 8}%
\special{pa 1200 520}%
\special{pa 1200 680}%
\special{dt 0.045}%
%
\special{pn 8}%
\special{pa 1600 520}%
\special{pa 1600 680}%
\special{dt 0.045}%
%
\special{pn 8}%
\special{pa 2000 520}%
\special{pa 2000 680}%
\special{dt 0.045}%
%
\special{pn 8}%
\special{pa 2720 520}%
\special{pa 2720 680}%
\special{dt 0.045}%
%
\special{pn 8}%
\special{pa 3100 800}%
\special{pa 3100 720}%
\special{fp}%
%
\special{pn 8}%
\special{pa 600 800}%
\special{pa 600 400}%
\special{fp}%
%
\special{pn 8}%
\special{pa 2920 400}%
\special{pa 2920 680}%
\special{fp}%
\end{picture}}%
\caption{the curve $\gamma_d$ in Example \ref{ex:even}}
\label{fig:gamma_even}
\end{figure}
\end{center}

\begin{proof}
We have $C(d)=C_{\gamma_d}$, where $\gamma_d$ is the curve as shown in Figure \ref{fig:gamma_even}.
Let $d\ge 4$ be an even integer. Then
\begin{align}
\label{eq:ex2}
& \quad 
{\unitlength 0.1in%
\begin{picture}(40.0000,7.2000)(4.0000,-9.6000)%
%
\special{pn 8}%
\special{ar 1900 680 100 120 6.2831853 3.1415927}%
%
\special{pn 8}%
\special{ar 2100 520 100 120 3.1415927 6.2831853}%
%
\special{pn 8}%
\special{ar 2500 520 100 120 3.1415927 6.2831853}%
%
\special{pn 8}%
\special{ar 2300 680 100 120 6.2831853 3.1415927}%
%
\special{pn 8}%
\special{ar 2700 680 100 120 6.2831853 3.1415927}%
%
\special{pn 8}%
\special{ar 3300 520 100 120 3.1415927 4.7123890}%
%
\special{pn 8}%
\special{ar 3620 520 100 120 4.7123890 6.2831853}%
%
\special{pn 8}%
\special{ar 3820 680 100 120 6.2831853 3.1415927}%
%
\special{pn 4}%
\special{sh 1}%
\special{ar 1200 600 16 16 0 6.2831853}%
%
\special{pn 8}%
\special{ar 1700 800 100 120 1.5707963 3.1415927}%
%
\special{pn 8}%
\special{ar 4000 800 100 120 6.2831853 1.5707963}%
%
\special{pn 8}%
\special{ar 4000 720 100 120 4.7123890 6.2831853}%
%
\special{pn 8}%
\special{pa 2000 680}%
\special{pa 2000 520}%
\special{fp}%
%
\special{pn 8}%
\special{pa 2400 680}%
\special{pa 2400 520}%
\special{fp}%
%
\special{pn 8}%
\special{pa 3200 680}%
\special{pa 3200 520}%
\special{fp}%
%
\special{pn 4}%
\special{sh 1}%
\special{ar 3920 600 16 16 0 6.2831853}%
%
\special{pn 4}%
\special{sh 1}%
\special{ar 2800 600 16 16 0 6.2831853}%
%
\special{pn 4}%
\special{sh 1}%
\special{ar 2400 600 16 16 0 6.2831853}%
%
\special{pn 4}%
\special{sh 1}%
\special{ar 2000 600 16 16 0 6.2831853}%
%
\special{pn 4}%
\special{sh 1}%
\special{ar 1600 600 16 16 0 6.2831853}%
\put(33.2000,-4.8000){\makebox(0,0)[lb]{$\cdots$}}%
%
\special{pn 8}%
\special{pa 1400 600}%
\special{pa 1320 560}%
\special{fp}%
\special{pa 1400 600}%
\special{pa 1320 640}%
\special{fp}%
%
\special{pn 8}%
\special{pa 2800 920}%
\special{pa 2880 880}%
\special{fp}%
\special{pa 2800 920}%
\special{pa 2880 960}%
\special{fp}%
\put(15.1000,-5.6000){\makebox(0,0)[lb]{$1$}}%
\put(19.1000,-5.6000){\makebox(0,0)[lb]{$2$}}%
\put(23.1000,-5.6000){\makebox(0,0)[lb]{$3$}}%
%
\special{pn 8}%
\special{ar 2900 520 100 120 3.1415927 6.2831853}%
%
\special{pn 8}%
\special{pa 2800 680}%
\special{pa 2800 520}%
\special{fp}%
\put(27.1000,-5.6000){\makebox(0,0)[lb]{$3$}}%
%
\special{pn 8}%
\special{ar 3100 680 100 120 6.2831853 3.1415927}%
%
\special{pn 4}%
\special{sh 1}%
\special{ar 3200 600 16 16 0 6.2831853}%
%
\special{pn 8}%
\special{pa 4000 920}%
\special{pa 1700 920}%
\special{fp}%
%
\special{pn 8}%
\special{pa 4000 600}%
\special{pa 1200 600}%
\special{fp}%
%
\special{pn 8}%
\special{ar 3820 400 100 120 4.7123890 6.2831853}%
%
\special{pn 8}%
\special{ar 1700 400 100 120 3.1415927 4.7123890}%
%
\special{pn 8}%
\special{ar 1700 520 100 120 4.7123890 6.2831853}%
%
\special{pn 4}%
\special{sh 1}%
\special{ar 1200 400 16 16 0 6.2831853}%
%
\special{pn 8}%
\special{pa 3820 280}%
\special{pa 1700 280}%
\special{fp}%
\put(39.5000,-5.6000){\makebox(0,0)[lb]{$d+1$}}%
%
\special{pn 8}%
\special{pa 1320 400}%
\special{pa 1400 360}%
\special{fp}%
\special{pa 1320 400}%
\special{pa 1400 440}%
\special{fp}%
%
\special{pn 8}%
\special{pa 2880 280}%
\special{pa 2800 240}%
\special{fp}%
\special{pa 2880 280}%
\special{pa 2800 320}%
\special{fp}%
%
\special{pn 4}%
\special{sh 1}%
\special{ar 1600 400 16 16 0 6.2831853}%
%
\special{pn 8}%
\special{pa 1700 400}%
\special{pa 1200 400}%
\special{fp}%
\put(4.0000,-6.7500){\makebox(0,0)[lb]{$\langle \gamma_{d+2} \rangle=$}}%
%
\special{pn 8}%
\special{pa 1800 520}%
\special{pa 1800 680}%
\special{dt 0.045}%
%
\special{pn 8}%
\special{pa 2200 520}%
\special{pa 2200 680}%
\special{dt 0.045}%
%
\special{pn 8}%
\special{pa 2600 520}%
\special{pa 2600 680}%
\special{dt 0.045}%
%
\special{pn 8}%
\special{pa 3000 520}%
\special{pa 3000 680}%
\special{dt 0.045}%
%
\special{pn 8}%
\special{pa 3720 520}%
\special{pa 3720 680}%
\special{dt 0.045}%
%
\special{pn 8}%
\special{pa 1600 800}%
\special{pa 1600 400}%
\special{fp}%
%
\special{pn 8}%
\special{pa 4100 800}%
\special{pa 4100 720}%
\special{fp}%
%
\special{pn 8}%
\special{pa 3920 400}%
\special{pa 3920 680}%
\special{fp}%
%
\special{pn 13}%
\special{pa 4200 280}%
\special{pa 4400 600}%
\special{fp}%
\special{pa 4400 600}%
\special{pa 4200 920}%
\special{fp}%
%
\special{pn 13}%
\special{pa 1200 280}%
\special{pa 1000 600}%
\special{fp}%
\special{pa 1000 600}%
\special{pa 1200 920}%
\special{fp}%
\end{picture}}%
 \nonumber \\
& \hspace{4em} 
{\unitlength 0.1in%
\begin{picture}(38.0000,6.4000)(2.0000,-9.2000)%
%
\special{pn 8}%
\special{ar 1500 680 100 120 6.2831853 3.1415927}%
%
\special{pn 8}%
\special{ar 1700 520 100 120 3.1415927 6.2831853}%
%
\special{pn 8}%
\special{ar 2100 520 100 120 3.1415927 6.2831853}%
%
\special{pn 8}%
\special{ar 1900 680 100 120 6.2831853 3.1415927}%
%
\special{pn 8}%
\special{ar 2300 680 100 120 6.2831853 3.1415927}%
%
\special{pn 8}%
\special{ar 2900 520 100 120 3.1415927 4.7123890}%
%
\special{pn 8}%
\special{ar 3220 520 100 120 4.7123890 6.2831853}%
%
\special{pn 8}%
\special{ar 3420 680 100 120 6.2831853 3.1415927}%
%
\special{pn 4}%
\special{sh 1}%
\special{ar 800 600 16 16 0 6.2831853}%
%
\special{pn 8}%
\special{ar 1300 800 100 120 1.5707963 3.1415927}%
%
\special{pn 8}%
\special{ar 3600 800 100 120 6.2831853 1.5707963}%
%
\special{pn 8}%
\special{ar 3600 720 100 120 4.7123890 6.2831853}%
%
\special{pn 4}%
\special{sh 1}%
\special{ar 3520 600 16 16 0 6.2831853}%
%
\special{pn 4}%
\special{sh 1}%
\special{ar 2400 600 16 16 0 6.2831853}%
%
\special{pn 4}%
\special{sh 1}%
\special{ar 2000 600 16 16 0 6.2831853}%
%
\special{pn 4}%
\special{sh 1}%
\special{ar 1200 600 16 16 0 6.2831853}%
\put(29.2000,-4.8000){\makebox(0,0)[lb]{$\cdots$}}%
\put(11.1000,-5.6000){\makebox(0,0)[lb]{$1$}}%
\put(19.1000,-5.6000){\makebox(0,0)[lb]{$3$}}%
%
\special{pn 8}%
\special{ar 2500 520 100 120 3.1415927 6.2831853}%
\put(23.1000,-5.6000){\makebox(0,0)[lb]{$4$}}%
%
\special{pn 8}%
\special{ar 2700 680 100 120 6.2831853 3.1415927}%
%
\special{pn 4}%
\special{sh 1}%
\special{ar 2800 600 16 16 0 6.2831853}%
%
\special{pn 8}%
\special{pa 3600 920}%
\special{pa 1300 920}%
\special{fp}%
%
\special{pn 8}%
\special{ar 3420 400 100 120 4.7123890 6.2831853}%
%
\special{pn 8}%
\special{ar 1300 400 100 120 3.1415927 4.7123890}%
%
\special{pn 8}%
\special{ar 1300 520 100 120 4.7123890 6.2831853}%
%
\special{pn 8}%
\special{pa 1300 400}%
\special{pa 1240 400}%
\special{fp}%
%
\special{pn 8}%
\special{pa 1160 400}%
\special{pa 800 400}%
\special{fp}%
%
\special{pn 4}%
\special{sh 1}%
\special{ar 800 400 16 16 0 6.2831853}%
%
\special{pn 8}%
\special{pa 3420 280}%
\special{pa 1300 280}%
\special{fp}%
\put(35.5000,-5.6000){\makebox(0,0)[lb]{$d+1$}}%
%
\special{pn 8}%
\special{ar 1520 520 80 80 6.2831853 1.5707963}%
%
\special{pn 8}%
\special{ar 1680 680 80 80 3.1415927 4.7123890}%
%
\special{pn 8}%
\special{pa 1520 600}%
\special{pa 800 600}%
\special{fp}%
%
\special{pn 8}%
\special{pa 1680 600}%
\special{pa 3600 600}%
\special{fp}%
\put(2.0000,-6.5000){\makebox(0,0)[lb]{$=A$}}%
%
\special{pn 4}%
\special{sh 1}%
\special{ar 1200 400 16 16 0 6.2831853}%
%
\special{pn 8}%
\special{pa 2000 680}%
\special{pa 2000 640}%
\special{fp}%
%
\special{pn 8}%
\special{pa 2000 560}%
\special{pa 2000 520}%
\special{fp}%
%
\special{pn 8}%
\special{pa 2400 680}%
\special{pa 2400 640}%
\special{fp}%
%
\special{pn 8}%
\special{pa 2400 560}%
\special{pa 2400 520}%
\special{fp}%
%
\special{pn 8}%
\special{pa 2800 680}%
\special{pa 2800 640}%
\special{fp}%
%
\special{pn 8}%
\special{pa 2800 560}%
\special{pa 2800 520}%
\special{fp}%
%
\special{pn 8}%
\special{pa 3520 680}%
\special{pa 3520 640}%
\special{fp}%
%
\special{pn 8}%
\special{pa 1400 520}%
\special{pa 1400 680}%
\special{dt 0.045}%
%
\special{pn 8}%
\special{pa 1800 520}%
\special{pa 1800 680}%
\special{dt 0.045}%
%
\special{pn 8}%
\special{pa 2200 520}%
\special{pa 2200 680}%
\special{dt 0.045}%
%
\special{pn 8}%
\special{pa 2600 520}%
\special{pa 2600 680}%
\special{dt 0.045}%
%
\special{pn 8}%
\special{pa 3320 520}%
\special{pa 3320 680}%
\special{dt 0.045}%
%
\special{pn 13}%
\special{pa 800 280}%
\special{pa 600 600}%
\special{fp}%
\special{pa 600 600}%
\special{pa 800 920}%
\special{fp}%
%
\special{pn 8}%
\special{pa 1200 400}%
\special{pa 1200 560}%
\special{fp}%
%
\special{pn 8}%
\special{pa 1200 800}%
\special{pa 1200 640}%
\special{fp}%
%
\special{pn 8}%
\special{pa 3700 800}%
\special{pa 3700 720}%
\special{fp}%
%
\special{pn 8}%
\special{pa 3520 400}%
\special{pa 3520 560}%
\special{fp}%
%
\special{pn 13}%
\special{pa 3800 280}%
\special{pa 4000 600}%
\special{fp}%
\special{pa 4000 600}%
\special{pa 3800 920}%
\special{fp}%
\end{picture}}%
 \nonumber \\
& \hspace{6em} 
{\unitlength 0.1in%
\begin{picture}(38.0000,6.4000)(2.0000,-9.2000)%
%
\special{pn 8}%
\special{ar 1500 680 100 120 6.2831853 3.1415927}%
%
\special{pn 8}%
\special{ar 1700 520 100 120 3.1415927 6.2831853}%
%
\special{pn 8}%
\special{ar 2100 520 100 120 3.1415927 6.2831853}%
%
\special{pn 8}%
\special{ar 1900 680 100 120 6.2831853 3.1415927}%
%
\special{pn 8}%
\special{ar 2300 680 100 120 6.2831853 3.1415927}%
%
\special{pn 8}%
\special{ar 2900 520 100 120 3.1415927 4.7123890}%
%
\special{pn 8}%
\special{ar 3220 520 100 120 4.7123890 6.2831853}%
%
\special{pn 8}%
\special{ar 3420 680 100 120 6.2831853 3.1415927}%
%
\special{pn 4}%
\special{sh 1}%
\special{ar 800 600 16 16 0 6.2831853}%
%
\special{pn 8}%
\special{ar 1300 800 100 120 1.5707963 3.1415927}%
%
\special{pn 8}%
\special{ar 3600 800 100 120 6.2831853 1.5707963}%
%
\special{pn 8}%
\special{ar 3600 720 100 120 4.7123890 6.2831853}%
%
\special{pn 4}%
\special{sh 1}%
\special{ar 3520 600 16 16 0 6.2831853}%
%
\special{pn 4}%
\special{sh 1}%
\special{ar 2400 600 16 16 0 6.2831853}%
%
\special{pn 4}%
\special{sh 1}%
\special{ar 2000 600 16 16 0 6.2831853}%
%
\special{pn 4}%
\special{sh 1}%
\special{ar 1200 600 16 16 0 6.2831853}%
\put(29.2000,-4.8000){\makebox(0,0)[lb]{$\cdots$}}%
\put(11.1000,-5.6000){\makebox(0,0)[lb]{$1$}}%
\put(19.1000,-5.6000){\makebox(0,0)[lb]{$3$}}%
%
\special{pn 8}%
\special{ar 2500 520 100 120 3.1415927 6.2831853}%
\put(23.1000,-5.6000){\makebox(0,0)[lb]{$4$}}%
%
\special{pn 8}%
\special{ar 2700 680 100 120 6.2831853 3.1415927}%
%
\special{pn 4}%
\special{sh 1}%
\special{ar 2800 600 16 16 0 6.2831853}%
%
\special{pn 8}%
\special{pa 3600 920}%
\special{pa 1300 920}%
\special{fp}%
%
\special{pn 8}%
\special{ar 3420 400 100 120 4.7123890 6.2831853}%
%
\special{pn 8}%
\special{ar 1300 400 100 120 3.1415927 4.7123890}%
%
\special{pn 8}%
\special{ar 1300 520 100 120 4.7123890 6.2831853}%
%
\special{pn 8}%
\special{pa 1300 400}%
\special{pa 1240 400}%
\special{fp}%
%
\special{pn 8}%
\special{pa 1160 400}%
\special{pa 800 400}%
\special{fp}%
%
\special{pn 4}%
\special{sh 1}%
\special{ar 800 400 16 16 0 6.2831853}%
%
\special{pn 8}%
\special{pa 3420 280}%
\special{pa 1300 280}%
\special{fp}%
\put(35.5000,-5.6000){\makebox(0,0)[lb]{$d+1$}}%
%
\special{pn 8}%
\special{pa 1520 600}%
\special{pa 800 600}%
\special{fp}%
%
\special{pn 8}%
\special{pa 1680 600}%
\special{pa 3600 600}%
\special{fp}%
%
\special{pn 8}%
\special{ar 1680 520 80 80 1.5707963 3.1415927}%
%
\special{pn 8}%
\special{ar 1520 680 80 80 4.7123890 6.2831853}%
\put(2.0000,-6.5000){\makebox(0,0)[lb]{$+A^{-1}$}}%
%
\special{pn 4}%
\special{sh 1}%
\special{ar 1200 400 16 16 0 6.2831853}%
%
\special{pn 8}%
\special{pa 2000 680}%
\special{pa 2000 640}%
\special{fp}%
%
\special{pn 8}%
\special{pa 2000 560}%
\special{pa 2000 520}%
\special{fp}%
%
\special{pn 8}%
\special{pa 2400 520}%
\special{pa 2400 560}%
\special{fp}%
%
\special{pn 8}%
\special{pa 2400 640}%
\special{pa 2400 680}%
\special{fp}%
%
\special{pn 8}%
\special{pa 2800 680}%
\special{pa 2800 640}%
\special{fp}%
%
\special{pn 8}%
\special{pa 2800 560}%
\special{pa 2800 520}%
\special{fp}%
%
\special{pn 8}%
\special{pa 3520 680}%
\special{pa 3520 640}%
\special{fp}%
%
\special{pn 8}%
\special{pa 1400 520}%
\special{pa 1400 680}%
\special{dt 0.045}%
%
\special{pn 8}%
\special{pa 1800 520}%
\special{pa 1800 680}%
\special{dt 0.045}%
%
\special{pn 8}%
\special{pa 2200 520}%
\special{pa 2200 680}%
\special{dt 0.045}%
%
\special{pn 8}%
\special{pa 2600 520}%
\special{pa 2600 680}%
\special{dt 0.045}%
%
\special{pn 8}%
\special{pa 3320 520}%
\special{pa 3320 680}%
\special{dt 0.045}%
%
\special{pn 13}%
\special{pa 800 280}%
\special{pa 600 600}%
\special{fp}%
\special{pa 600 600}%
\special{pa 800 920}%
\special{fp}%
%
\special{pn 8}%
\special{pa 1200 400}%
\special{pa 1200 560}%
\special{fp}%
%
\special{pn 8}%
\special{pa 1200 640}%
\special{pa 1200 800}%
\special{fp}%
%
\special{pn 8}%
\special{pa 3700 800}%
\special{pa 3700 720}%
\special{fp}%
%
\special{pn 8}%
\special{pa 3520 400}%
\special{pa 3520 560}%
\special{fp}%
%
\special{pn 13}%
\special{pa 3800 280}%
\special{pa 4000 600}%
\special{fp}%
\special{pa 4000 600}%
\special{pa 3800 920}%
\special{fp}%
\end{picture}}%
  \nonumber \\
\end{align}
The diagram in the first term of (\ref{eq:ex2}) can be expanded as
\begin{align*}
& \quad 
{\unitlength 0.1in%
\begin{picture}(36.0000,6.4000)(4.0000,-9.2000)%
%
\special{pn 8}%
\special{ar 1500 680 100 120 6.2831853 3.1415927}%
%
\special{pn 8}%
\special{ar 1700 520 100 120 3.1415927 6.2831853}%
%
\special{pn 8}%
\special{ar 2100 520 100 120 3.1415927 6.2831853}%
%
\special{pn 8}%
\special{ar 1900 680 100 120 6.2831853 3.1415927}%
%
\special{pn 8}%
\special{ar 2300 680 100 120 6.2831853 3.1415927}%
%
\special{pn 8}%
\special{ar 2900 520 100 120 3.1415927 4.7123890}%
%
\special{pn 8}%
\special{ar 3220 520 100 120 4.7123890 6.2831853}%
%
\special{pn 8}%
\special{ar 3420 680 100 120 6.2831853 3.1415927}%
%
\special{pn 4}%
\special{sh 1}%
\special{ar 800 600 16 16 0 6.2831853}%
%
\special{pn 8}%
\special{ar 1300 800 100 120 1.5707963 3.1415927}%
%
\special{pn 8}%
\special{ar 3600 800 100 120 6.2831853 1.5707963}%
%
\special{pn 8}%
\special{ar 3600 720 100 120 4.7123890 6.2831853}%
%
\special{pn 4}%
\special{sh 1}%
\special{ar 3520 600 16 16 0 6.2831853}%
%
\special{pn 4}%
\special{sh 1}%
\special{ar 2400 600 16 16 0 6.2831853}%
%
\special{pn 4}%
\special{sh 1}%
\special{ar 1200 600 16 16 0 6.2831853}%
\put(29.2000,-4.8000){\makebox(0,0)[lb]{$\cdots$}}%
\put(11.1000,-5.6000){\makebox(0,0)[lb]{$1$}}%
%
\special{pn 8}%
\special{ar 2500 520 100 120 3.1415927 6.2831853}%
\put(23.1000,-5.6000){\makebox(0,0)[lb]{$4$}}%
%
\special{pn 8}%
\special{ar 2700 680 100 120 6.2831853 3.1415927}%
%
\special{pn 4}%
\special{sh 1}%
\special{ar 2800 600 16 16 0 6.2831853}%
%
\special{pn 8}%
\special{pa 3600 920}%
\special{pa 1300 920}%
\special{fp}%
%
\special{pn 8}%
\special{ar 3420 400 100 120 4.7123890 6.2831853}%
%
\special{pn 8}%
\special{ar 1300 400 100 120 3.1415927 4.7123890}%
%
\special{pn 8}%
\special{ar 1300 520 100 120 4.7123890 6.2831853}%
%
\special{pn 8}%
\special{pa 1300 400}%
\special{pa 1240 400}%
\special{fp}%
%
\special{pn 8}%
\special{pa 1160 400}%
\special{pa 800 400}%
\special{fp}%
%
\special{pn 4}%
\special{sh 1}%
\special{ar 800 400 16 16 0 6.2831853}%
%
\special{pn 8}%
\special{pa 3420 280}%
\special{pa 1300 280}%
\special{fp}%
\put(35.5000,-5.6000){\makebox(0,0)[lb]{$d+1$}}%
%
\special{pn 8}%
\special{ar 1520 520 80 80 6.2831853 1.5707963}%
%
\special{pn 8}%
\special{ar 1680 680 80 80 3.1415927 4.7123890}%
%
\special{pn 8}%
\special{pa 1520 600}%
\special{pa 800 600}%
\special{fp}%
\put(4.0000,-6.5000){\makebox(0,0)[lb]{$A$}}%
%
\special{pn 8}%
\special{ar 1920 520 80 80 6.2831853 1.5707963}%
%
\special{pn 8}%
\special{ar 2080 680 80 80 3.1415927 4.7123890}%
%
\special{pn 8}%
\special{pa 1680 600}%
\special{pa 1920 600}%
\special{fp}%
%
\special{pn 8}%
\special{pa 2080 600}%
\special{pa 3600 600}%
\special{fp}%
%
\special{pn 4}%
\special{sh 1}%
\special{ar 1200 400 16 16 0 6.2831853}%
%
\special{pn 8}%
\special{pa 2400 680}%
\special{pa 2400 640}%
\special{fp}%
%
\special{pn 8}%
\special{pa 2400 560}%
\special{pa 2400 520}%
\special{fp}%
%
\special{pn 8}%
\special{pa 2800 520}%
\special{pa 2800 560}%
\special{fp}%
%
\special{pn 8}%
\special{pa 2800 640}%
\special{pa 2800 680}%
\special{fp}%
%
\special{pn 8}%
\special{pa 3520 680}%
\special{pa 3520 640}%
\special{fp}%
%
\special{pn 8}%
\special{pa 1400 520}%
\special{pa 1400 680}%
\special{dt 0.045}%
%
\special{pn 8}%
\special{pa 1800 520}%
\special{pa 1800 680}%
\special{dt 0.045}%
%
\special{pn 8}%
\special{pa 2200 520}%
\special{pa 2200 680}%
\special{dt 0.045}%
%
\special{pn 8}%
\special{pa 2600 520}%
\special{pa 2600 680}%
\special{dt 0.045}%
%
\special{pn 8}%
\special{pa 3320 520}%
\special{pa 3320 680}%
\special{dt 0.045}%
%
\special{pn 13}%
\special{pa 800 280}%
\special{pa 600 600}%
\special{fp}%
\special{pa 600 600}%
\special{pa 800 920}%
\special{fp}%
%
\special{pn 8}%
\special{pa 1200 400}%
\special{pa 1200 560}%
\special{fp}%
%
\special{pn 8}%
\special{pa 1200 640}%
\special{pa 1200 800}%
\special{fp}%
%
\special{pn 8}%
\special{pa 3520 400}%
\special{pa 3520 560}%
\special{fp}%
%
\special{pn 8}%
\special{pa 3700 800}%
\special{pa 3700 720}%
\special{fp}%
%
\special{pn 13}%
\special{pa 3800 280}%
\special{pa 4000 600}%
\special{fp}%
\special{pa 4000 600}%
\special{pa 3800 920}%
\special{fp}%
\end{picture}}%
 \\
& \hspace{2em} \quad 
{\unitlength 0.1in%
\begin{picture}(38.0000,6.4000)(2.0000,-9.2000)%
%
\special{pn 8}%
\special{ar 1500 680 100 120 6.2831853 3.1415927}%
%
\special{pn 8}%
\special{ar 1700 520 100 120 3.1415927 6.2831853}%
%
\special{pn 8}%
\special{ar 2100 520 100 120 3.1415927 6.2831853}%
%
\special{pn 8}%
\special{ar 1900 680 100 120 6.2831853 3.1415927}%
%
\special{pn 8}%
\special{ar 2300 680 100 120 6.2831853 3.1415927}%
%
\special{pn 8}%
\special{ar 2900 520 100 120 3.1415927 4.7123890}%
%
\special{pn 8}%
\special{ar 3220 520 100 120 4.7123890 6.2831853}%
%
\special{pn 8}%
\special{ar 3420 680 100 120 6.2831853 3.1415927}%
%
\special{pn 4}%
\special{sh 1}%
\special{ar 800 600 16 16 0 6.2831853}%
%
\special{pn 8}%
\special{ar 1300 800 100 120 1.5707963 3.1415927}%
%
\special{pn 8}%
\special{ar 3600 800 100 120 6.2831853 1.5707963}%
%
\special{pn 8}%
\special{ar 3600 720 100 120 4.7123890 6.2831853}%
%
\special{pn 4}%
\special{sh 1}%
\special{ar 3520 600 16 16 0 6.2831853}%
%
\special{pn 4}%
\special{sh 1}%
\special{ar 2400 600 16 16 0 6.2831853}%
%
\special{pn 4}%
\special{sh 1}%
\special{ar 1200 600 16 16 0 6.2831853}%
\put(29.2000,-4.8000){\makebox(0,0)[lb]{$\cdots$}}%
\put(11.1000,-5.6000){\makebox(0,0)[lb]{$1$}}%
%
\special{pn 8}%
\special{ar 2500 520 100 120 3.1415927 6.2831853}%
\put(23.1000,-5.6000){\makebox(0,0)[lb]{$4$}}%
%
\special{pn 8}%
\special{ar 2700 680 100 120 6.2831853 3.1415927}%
%
\special{pn 4}%
\special{sh 1}%
\special{ar 2800 600 16 16 0 6.2831853}%
%
\special{pn 8}%
\special{pa 3600 920}%
\special{pa 1300 920}%
\special{fp}%
%
\special{pn 8}%
\special{ar 3420 400 100 120 4.7123890 6.2831853}%
%
\special{pn 8}%
\special{ar 1300 400 100 120 3.1415927 4.7123890}%
%
\special{pn 8}%
\special{ar 1300 520 100 120 4.7123890 6.2831853}%
%
\special{pn 8}%
\special{pa 1300 400}%
\special{pa 1240 400}%
\special{fp}%
%
\special{pn 8}%
\special{pa 1160 400}%
\special{pa 800 400}%
\special{fp}%
%
\special{pn 4}%
\special{sh 1}%
\special{ar 800 400 16 16 0 6.2831853}%
%
\special{pn 8}%
\special{pa 3420 280}%
\special{pa 1300 280}%
\special{fp}%
\put(35.5000,-5.6000){\makebox(0,0)[lb]{$d+1$}}%
%
\special{pn 8}%
\special{ar 1520 520 80 80 6.2831853 1.5707963}%
%
\special{pn 8}%
\special{ar 1680 680 80 80 3.1415927 4.7123890}%
%
\special{pn 8}%
\special{pa 1520 600}%
\special{pa 800 600}%
\special{fp}%
\put(2.0000,-6.5000){\makebox(0,0)[lb]{$+A^{-1}$}}%
%
\special{pn 8}%
\special{pa 1680 600}%
\special{pa 1920 600}%
\special{fp}%
%
\special{pn 8}%
\special{pa 2080 600}%
\special{pa 3600 600}%
\special{fp}%
%
\special{pn 8}%
\special{ar 2080 520 80 80 1.5707963 3.1415927}%
%
\special{pn 8}%
\special{ar 1920 680 80 80 4.7123890 6.2831853}%
%
\special{pn 4}%
\special{sh 1}%
\special{ar 1200 400 16 16 0 6.2831853}%
%
\special{pn 8}%
\special{pa 2400 680}%
\special{pa 2400 640}%
\special{fp}%
%
\special{pn 8}%
\special{pa 2400 560}%
\special{pa 2400 520}%
\special{fp}%
%
\special{pn 8}%
\special{pa 2800 680}%
\special{pa 2800 640}%
\special{fp}%
%
\special{pn 8}%
\special{pa 2800 560}%
\special{pa 2800 520}%
\special{fp}%
%
\special{pn 8}%
\special{pa 3520 680}%
\special{pa 3520 640}%
\special{fp}%
%
\special{pn 8}%
\special{pa 1400 520}%
\special{pa 1400 680}%
\special{dt 0.045}%
%
\special{pn 8}%
\special{pa 1800 520}%
\special{pa 1800 680}%
\special{dt 0.045}%
%
\special{pn 8}%
\special{pa 2200 520}%
\special{pa 2200 680}%
\special{dt 0.045}%
%
\special{pn 8}%
\special{pa 2600 520}%
\special{pa 2600 680}%
\special{dt 0.045}%
%
\special{pn 8}%
\special{pa 3320 520}%
\special{pa 3320 680}%
\special{dt 0.045}%
%
\special{pn 13}%
\special{pa 800 280}%
\special{pa 600 600}%
\special{fp}%
\special{pa 600 600}%
\special{pa 800 920}%
\special{fp}%
%
\special{pn 8}%
\special{pa 1200 400}%
\special{pa 1200 560}%
\special{fp}%
%
\special{pn 8}%
\special{pa 1200 640}%
\special{pa 1200 800}%
\special{fp}%
%
\special{pn 8}%
\special{pa 3520 400}%
\special{pa 3520 560}%
\special{fp}%
%
\special{pn 8}%
\special{pa 3700 800}%
\special{pa 3700 720}%
\special{fp}%
%
\special{pn 13}%
\special{pa 3800 280}%
\special{pa 4000 600}%
\special{fp}%
\special{pa 4000 600}%
\special{pa 3800 920}%
\special{fp}%
\end{picture}}%
 \\
& 
{\unitlength 0.1in%
\begin{picture}(27.8000,3.6000)(4.8000,-6.2000)%
%
\special{pn 4}%
\special{sh 1}%
\special{ar 2400 500 16 16 0 6.2831853}%
%
\special{pn 8}%
\special{ar 2900 530 100 90 1.5707963 3.1415927}%
%
\special{pn 4}%
\special{sh 1}%
\special{ar 2800 500 16 16 0 6.2831853}%
%
\special{pn 8}%
\special{ar 3020 350 100 90 4.7123890 6.2831853}%
%
\special{pn 8}%
\special{ar 2900 350 100 90 3.1415927 4.7123890}%
%
\special{pn 8}%
\special{pa 2900 350}%
\special{pa 2840 350}%
\special{fp}%
%
\special{pn 8}%
\special{pa 2760 350}%
\special{pa 2400 350}%
\special{fp}%
%
\special{pn 4}%
\special{sh 1}%
\special{ar 2400 350 16 16 0 6.2831853}%
\put(4.8000,-4.8100){\makebox(0,0)[lb]{$=A \langle \gamma_d \rangle+A^{-1}\cdot (-A^{-3})^{d-2}$}}%
%
\special{pn 8}%
\special{ar 2900 425 80 75 4.7123890 1.5707963}%
%
\special{pn 8}%
\special{pa 2900 500}%
\special{pa 2400 500}%
\special{fp}%
%
\special{pn 8}%
\special{ar 3020 530 100 90 6.2831853 1.5707963}%
%
\special{pn 8}%
\special{pa 3020 620}%
\special{pa 2900 620}%
\special{fp}%
%
\special{pn 8}%
\special{pa 2900 260}%
\special{pa 3020 260}%
\special{fp}%
%
\special{pn 4}%
\special{sh 1}%
\special{ar 2800 350 16 16 0 6.2831853}%
%
\special{pn 8}%
\special{pa 2800 350}%
\special{pa 2800 470}%
\special{fp}%
%
\special{pn 8}%
\special{pa 3120 350}%
\special{pa 3120 530}%
\special{fp}%
%
\special{pn 13}%
\special{pa 2400 260}%
\special{pa 2300 440}%
\special{fp}%
\special{pa 2300 440}%
\special{pa 2400 620}%
\special{fp}%
%
\special{pn 13}%
\special{pa 3160 260}%
\special{pa 3260 440}%
\special{fp}%
\special{pa 3260 440}%
\special{pa 3160 620}%
\special{fp}%
\end{picture}}%
\end{align*}
On the other hand, the second term of (\ref{eq:ex2}) is equal to
\[
{\unitlength 0.1in%
\begin{picture}(21.0000,3.6000)(2.0000,-4.6000)%
\put(2.0000,-3.3800){\makebox(0,0)[lb]{$A^{-1}\cdot (-A^{-3})^{d-1}$}}%
%
\special{pn 8}%
\special{pa 1940 100}%
\special{pa 2060 100}%
\special{fp}%
%
\special{pn 4}%
\special{sh 1}%
\special{ar 1440 340 16 16 0 6.2831853}%
%
\special{pn 8}%
\special{ar 1940 370 100 90 1.5707963 3.1415927}%
%
\special{pn 4}%
\special{sh 1}%
\special{ar 1840 340 16 16 0 6.2831853}%
%
\special{pn 8}%
\special{ar 2060 190 100 90 4.7123890 6.2831853}%
%
\special{pn 8}%
\special{ar 1940 190 100 90 3.1415927 4.7123890}%
%
\special{pn 8}%
\special{pa 1940 190}%
\special{pa 1880 190}%
\special{fp}%
%
\special{pn 8}%
\special{pa 1800 190}%
\special{pa 1440 190}%
\special{fp}%
%
\special{pn 4}%
\special{sh 1}%
\special{ar 1440 190 16 16 0 6.2831853}%
%
\special{pn 8}%
\special{ar 1940 265 80 75 4.7123890 1.5707963}%
%
\special{pn 8}%
\special{pa 1940 340}%
\special{pa 1440 340}%
\special{fp}%
%
\special{pn 8}%
\special{ar 2060 370 100 90 6.2831853 1.5707963}%
%
\special{pn 8}%
\special{pa 2060 460}%
\special{pa 1940 460}%
\special{fp}%
%
\special{pn 4}%
\special{sh 1}%
\special{ar 1840 190 16 16 0 6.2831853}%
%
\special{pn 8}%
\special{pa 1840 190}%
\special{pa 1840 310}%
\special{fp}%
%
\special{pn 8}%
\special{pa 2160 190}%
\special{pa 2160 370}%
\special{fp}%
%
\special{pn 13}%
\special{pa 1440 100}%
\special{pa 1340 280}%
\special{fp}%
\special{pa 1340 280}%
\special{pa 1440 460}%
\special{fp}%
%
\special{pn 13}%
\special{pa 2200 100}%
\special{pa 2300 280}%
\special{fp}%
\special{pa 2300 280}%
\special{pa 2200 460}%
\special{fp}%
\end{picture}}%
\]
Moreover, we compute
\begin{align*}
& 
{\unitlength 0.1in%
\begin{picture}(39.6000,3.6000)(3.6000,-6.0000)%
%
\special{pn 4}%
\special{sh 1}%
\special{ar 1840 500 16 16 0 6.2831853}%
%
\special{pn 8}%
\special{ar 2460 330 100 90 4.7123890 6.2831853}%
%
\special{pn 8}%
\special{ar 2340 330 100 90 3.1415927 4.7123890}%
%
\special{pn 8}%
\special{pa 2340 350}%
\special{pa 2280 350}%
\special{fp}%
%
\special{pn 8}%
\special{pa 2200 350}%
\special{pa 1840 350}%
\special{fp}%
%
\special{pn 4}%
\special{sh 1}%
\special{ar 1840 350 16 16 0 6.2831853}%
%
\special{pn 8}%
\special{ar 2340 425 80 75 4.7123890 1.5707963}%
%
\special{pn 8}%
\special{ar 2460 510 100 90 6.2831853 1.5707963}%
%
\special{pn 8}%
\special{pa 2460 600}%
\special{pa 2340 600}%
\special{fp}%
%
\special{pn 8}%
\special{pa 2340 240}%
\special{pa 2460 240}%
\special{fp}%
%
\special{pn 4}%
\special{sh 1}%
\special{ar 2240 350 16 16 0 6.2831853}%
%
\special{pn 8}%
\special{ar 2320 560 80 60 3.1415927 4.7123890}%
%
\special{pn 8}%
\special{pa 2320 500}%
\special{pa 2340 500}%
\special{fp}%
%
\special{pn 8}%
\special{ar 2160 440 80 60 6.2831853 1.5707963}%
%
\special{pn 8}%
\special{pa 2160 500}%
\special{pa 1840 500}%
\special{fp}%
%
\special{pn 4}%
\special{sh 1}%
\special{ar 3460 500 16 16 0 6.2831853}%
%
\special{pn 8}%
\special{ar 4080 330 100 90 4.7123890 6.2831853}%
%
\special{pn 8}%
\special{ar 3960 330 100 90 3.1415927 4.7123890}%
%
\special{pn 8}%
\special{pa 3960 350}%
\special{pa 3900 350}%
\special{fp}%
%
\special{pn 8}%
\special{pa 3820 350}%
\special{pa 3460 350}%
\special{fp}%
%
\special{pn 4}%
\special{sh 1}%
\special{ar 3460 350 16 16 0 6.2831853}%
%
\special{pn 8}%
\special{ar 3960 425 80 75 4.7123890 1.5707963}%
%
\special{pn 8}%
\special{ar 4080 510 100 90 6.2831853 1.5707963}%
%
\special{pn 8}%
\special{pa 4080 600}%
\special{pa 3960 600}%
\special{fp}%
%
\special{pn 8}%
\special{pa 3960 240}%
\special{pa 4080 240}%
\special{fp}%
%
\special{pn 4}%
\special{sh 1}%
\special{ar 3860 350 16 16 0 6.2831853}%
%
\special{pn 8}%
\special{pa 3940 500}%
\special{pa 3960 500}%
\special{fp}%
%
\special{pn 8}%
\special{pa 3780 500}%
\special{pa 3460 500}%
\special{fp}%
%
\special{pn 8}%
\special{ar 3780 560 80 60 4.7123890 6.2831853}%
%
\special{pn 8}%
\special{ar 3940 440 80 60 1.5707963 3.1415927}%
\put(14.0000,-4.6200){\makebox(0,0)[lb]{$=A$}}%
\put(28.2000,-4.6200){\makebox(0,0)[lb]{$+\quad A^{-1}$}}%
%
\special{pn 4}%
\special{sh 1}%
\special{ar 460 480 16 16 0 6.2831853}%
%
\special{pn 8}%
\special{ar 960 510 100 90 1.5707963 3.1415927}%
%
\special{pn 4}%
\special{sh 1}%
\special{ar 860 480 16 16 0 6.2831853}%
%
\special{pn 8}%
\special{ar 1080 330 100 90 4.7123890 6.2831853}%
%
\special{pn 8}%
\special{ar 960 330 100 90 3.1415927 4.7123890}%
%
\special{pn 8}%
\special{pa 960 330}%
\special{pa 900 330}%
\special{fp}%
%
\special{pn 8}%
\special{pa 820 330}%
\special{pa 460 330}%
\special{fp}%
%
\special{pn 4}%
\special{sh 1}%
\special{ar 460 330 16 16 0 6.2831853}%
%
\special{pn 8}%
\special{ar 960 405 80 75 4.7123890 1.5707963}%
%
\special{pn 8}%
\special{pa 960 480}%
\special{pa 460 480}%
\special{fp}%
%
\special{pn 8}%
\special{ar 1080 510 100 90 6.2831853 1.5707963}%
%
\special{pn 8}%
\special{pa 1080 600}%
\special{pa 960 600}%
\special{fp}%
%
\special{pn 4}%
\special{sh 1}%
\special{ar 860 330 16 16 0 6.2831853}%
%
\special{pn 8}%
\special{pa 860 330}%
\special{pa 860 450}%
\special{fp}%
%
\special{pn 8}%
\special{pa 1180 330}%
\special{pa 1180 510}%
\special{fp}%
%
\special{pn 13}%
\special{pa 460 240}%
\special{pa 360 420}%
\special{fp}%
\special{pa 360 420}%
\special{pa 460 600}%
\special{fp}%
%
\special{pn 13}%
\special{pa 1220 240}%
\special{pa 1320 420}%
\special{fp}%
\special{pa 1320 420}%
\special{pa 1220 600}%
\special{fp}%
%
\special{pn 8}%
\special{pa 960 240}%
\special{pa 1080 240}%
\special{fp}%
%
\special{pn 13}%
\special{pa 1840 240}%
\special{pa 1740 420}%
\special{fp}%
\special{pa 1740 420}%
\special{pa 1840 600}%
\special{fp}%
%
\special{pn 13}%
\special{pa 2600 240}%
\special{pa 2700 420}%
\special{fp}%
\special{pa 2700 420}%
\special{pa 2600 600}%
\special{fp}%
%
\special{pn 8}%
\special{ar 2340 560 100 40 1.5707963 3.1415927}%
%
\special{pn 8}%
\special{pa 2560 510}%
\special{pa 2560 330}%
\special{fp}%
%
\special{pn 8}%
\special{pa 2240 330}%
\special{pa 2240 440}%
\special{fp}%
%
\special{pn 13}%
\special{pa 3460 240}%
\special{pa 3360 420}%
\special{fp}%
\special{pa 3360 420}%
\special{pa 3460 600}%
\special{fp}%
%
\special{pn 13}%
\special{pa 4220 240}%
\special{pa 4320 420}%
\special{fp}%
\special{pa 4320 420}%
\special{pa 4220 600}%
\special{fp}%
%
\special{pn 8}%
\special{pa 3860 330}%
\special{pa 3860 440}%
\special{fp}%
%
\special{pn 8}%
\special{ar 3780 560 80 60 4.7123890 6.2831853}%
%
\special{pn 8}%
\special{ar 3960 560 100 40 1.5707963 3.1415927}%
%
\special{pn 8}%
\special{pa 4180 510}%
\special{pa 4180 330}%
\special{fp}%
\end{picture}}%
 \\
=& A(-A^3)+A^{-1}(-A^{-3})=-A^{-4}-A^4.
\end{align*}
Therefore, we obtain
\begin{align*}
\langle \gamma_{d+2} \rangle
=& A^2 \langle \gamma_d \rangle+((-A^{-3})^{d-2}+A^{-1}\cdot (-A^{-3})^{d-1})
(-A^{-4}-A^4) \\
=& A^2 \langle \gamma_d \rangle+
A^{-3d-2}-A^{-3d+2}+A^{-3d+6}-A^{-3d+10}.
\end{align*}

Now, by a direct computation, we see that
$\langle \gamma_4 \rangle=A^{-8}-A^{-4}+1-A^4+A^8$,
and the formula is proved by an inductive argument.
\end{proof}

\begin{thm}
Let $d \ge 4$.
Any even integer $l\in \{0,6,8,\ldots,4d\}$ is $d$-realizable.
\end{thm}

\begin{proof}
Let $d=4$.
First, $0$, $6$, $10$, and $12$ are $4$-realizable.
To see this, for any $l\in \{0,6,10,12\}$ pick an oriented chord diagram $D$ of $3$ chords with $\spn D=l$; then the stacking $C=C_1\sharp D$ satisfies $\spn \langle C\rangle=l$.
Next, $8$ is $4$-realizable, since $C_4=\{(1,4),(2,7),(3,5),(6,8)\}$ satisfies $\langle C_4\rangle=A^4+A^2+1-A^{-2}-A^{-4}$, so that $\spn \langle C_4\rangle=8$; also $14$ is $4$-realizable, since $C'_4=\{(1,5),(2,4),(3,7),(6,8)\}$ satisfies $\langle C'_4 \rangle=A^8+A^6-A^4-A^2+1+A^{-2}-A^{-6}$, so that $\spn \langle C'_4 \rangle=14$.
Finally, the element $C(4)$ in Example \ref{ex:even} satisfies $\spn \langle C(4) \rangle=16$.

Now let $d\ge 5$ and assume that any even integer $l\in \{0,6,8,\ldots,4(d-1)\}$ is $(d-1)$-realizable.
Then by considering the stacking of $C_1$ and oriented linear chord diagrams of $d-1$ chords, we see that any $l\in \{0,6,8,\ldots,4(d-1)\}$ is $d$-realizable.
Next, the element $C(d)$ in Examples \ref{ex:odd} and \ref{ex:even} satisfies $\spn \langle C(d) \rangle=4d$.
Finally, the stacking $C=C_2\sharp C(d-2)$ satisfies $\spn \langle C \rangle=6+4(d-2)=4d-2$.

By induction on $d$, we obtain the assertion.
\end{proof}

It can be checked that $2$ and $4$ are not $d$-realizable for $d\le 6$.

\begin{quest}
Is there an oriented linear chord diagram $C$ such that $\spn \langle C\rangle=2$ or $4$?
\end{quest}

Finally, we study the case where the equality $\spn \langle C \rangle=4d$ holds.
This class of linear chord diagrams might be of interest
since it is closed under stacking by Proposition \ref{prop:stacking}.
Table 1 shows the number of linear chord diagrams
with $\spn \langle C \rangle=4d$ for a fixed integer $d\le 9$. 
To give another motivation, let us recall the following
classical result on characterization of alternating knots.

\begin{table}
\label{tbl:spn}
\caption{linear chord diagrams with $\spn \langle C \rangle=4d$}
\begin{tabular}{c|c|c|c|c|c|c|c|c}
$d$ & 2 & 3 & 4 & 5 & 6 & 7 & 8 & 9 \\
\hline
 & 0 & 2 & 4 & 12 & 84 & 338 & 1588 & 8588
\end{tabular}
\end{table}

\begin{thm}[\cite{KA1} \cite{MU1} \cite{Th87}]
Let $K$ be an oriented knot in $S^3$ and assume that the span of
the Jones polynomial $V_K(t)=\LL_K(t^{-1/4})$ is equal
to the minimum number of double points among all projection diagrams of $K$.
Then $K$ is alternating.
\end{thm}

Let $C$ be an oriented linear chord diagram of $d$ chords,
and let $\gamma$ be a curve on an oriented surface $S$ such that
$C_{\gamma}=C$.
Let $s_A$ and $s_B$ be the states of $D_{\gamma}$ which appeared in the proof of Proposition \ref{prop4.1}.
Suppose that $\spn \langle C\rangle=4d$.
Then, as we see from the proof of Proposition \ref{prop4.1},
we have
\begin{equation}
\label{eq:sasb}
\mu(s_A)+\mu(s_B)=d+2.
\end{equation}

\begin{rem}
The condition (\ref{eq:sasb}) does not imply $\spn  \langle C\rangle=4d$.
For example, let $C=\{ (1,8), (2,5), (3,6), (4,7)\}$.
Then $\mu(s_A)+\mu(s_B)=6$.
However, $\langle C\rangle=A^8+1-A^{-4}$ and
$\spn \langle C\rangle=12\neq 16$.
\end{rem}

Let us consider the following condition for $C$.
\begin{equation}
\label{eq:parity}
\text{for any chord $(i,j)$ of $C$, the parity of $i$ and $j$ are different.}
\end{equation}

\begin{thm}
\label{thm:i+j}
Keep the notation as above.
Then condition $(\ref{eq:sasb})$ implies condition $(\ref{eq:parity})$.
In particular, if $\spn \langle C \rangle=4d$, then
condition $(\ref{eq:parity})$ holds.
\end{thm}

To prove this, let us consider the following preliminary construction.
Let $N$ be a regular neighborhood of $\gamma$ in $S$.
We modify $N$ in a neighborhood of every double point of $\gamma$ by inserting two half-twisted bands as illustrated in Figure 16.
The result is denoted by $N'$, in which
the curve $\gamma$ embeds naturally.
Next, we give a labelling $A$ or $B$ to each boundary component
of a neighborhood in $N'$ of each double point as shown in Figure 16.
Then, this labelling extends naturally to a locally constant function
$\phi\colon \partial N' \setminus \partial S \to \{ A,B\}$.
From the construction, we see that the inverse image $\phi^{-1}(A)$ (resp. $\phi^{-1}(B)$) is homeomorphic to the splice of $D_{\gamma}$ by $s_A$ (resp. $s_B$).
Therefore, if we denote by $r'$ the number of boundary components of $N'$, we have
\begin{equation}
\label{eq:r'}
r'=\mu(s_A)+\mu(s_B)-1.
\end{equation}

\begin{center}
\label{fig:Nprime}
\begin{figure}
\input{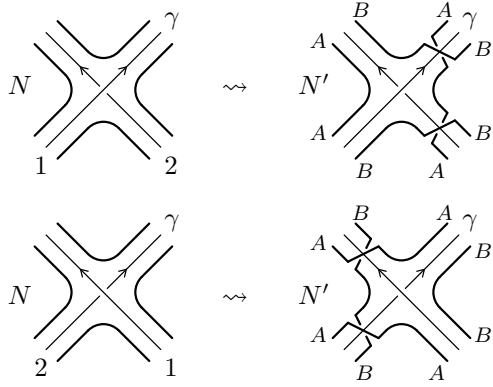}
\caption{the construction of the surface $N'$}
\end{figure}
\end{center}

\begin{lem}
\label{lem:orientable}
The surface $N'$ is orientable if and only if
condition $(\ref{eq:parity})$ holds.
\end{lem}

\begin{proof}
Let $\{p_k\}_{k=1}^d$ be the set of double points of $\gamma$,
and fix a parametrization $\gamma \colon I\to S$.
For each $k$, write $\gamma^{-1}(p_k)=\{t^1_k,t^2_k\}$ so that
$t^1_k<t^2_k$, and let $c_k\in H_1(N';\mathbb{Z})$ be
the homology class of the loop defined as the restriction of
$\gamma$ to $[t^1_k,t^2_k]$.
Then, the set $\{ c_k\}_{k=1}^d$ constitutes a $\mathbb{Z}$-basis for $H_1(N';\mathbb{Z})$.

Now, let $w_1\in H^1(N';\mathbb{Z}_2)\cong
{\rm Hom}(H_1(N';\mathbb{Z}), \mathbb{Z}_2)$ be the first
Stiefel-Whitney class of the tangent bundle of $N'$.
Let $(i_k,j_k)$ be the chord of $C$ corresponding to $p_k$.
Then, $w_1(c_k)$ is just the number of inserted half-twisted bands along the representative of $c_k$, and this is equal to
$|i_k-j_k|+1$.
Since $N'$ is orientable if and only if $w_1=0$, the assertion follows.
\end{proof}

\begin{proof}[Proof of Theorem \ref{thm:i+j}]
Since $N'$ is homotopy equivalent to the bouquet of $d$ circles,
the Euler characteristic of $N'$ is $\chi(N')=1-d$.

Assume that $N'$ is unorientable.
Since $r'=d+1$ from (\ref{eq:sasb}) and (\ref{eq:r'}),
there exists an integer $g>0$ and
$N'$ is homeomorphic to a connected sum of $g$ copies of $\mathbb{R}P^2$ minus the interior of $d+1$ disjoint union of closed disks.
Hence
\[
\chi(N')=2-g-(d+1)=1-g-d.
\]
Thus $g=0$, a contradiction.
Therefore $N'$ is orientable, and the conclusion
follows by Lemma \ref{lem:orientable}.
\end{proof}

\vspace{1em}

\noindent \textbf{Acknowledgements.}
The authors would like to thank Haruko Miyazawa and Akira Yasuhara for their helpful comments.
S. F. was supported by JSPS KAKENHI 26400098.
Y. K. was supported by JSPS KAKENHI 26800044.

\end{document}